\documentclass[a4,10pt,onesided]{article}
\usepackage[english]{babel}
\usepackage[T1]{fontenc}
\usepackage[utf8]{inputenc}
\usepackage{fancyhdr}
\usepackage{amsfonts}
\usepackage{amssymb}
\usepackage{amsmath}
\usepackage{amsthm}
\usepackage{mathtools}
\usepackage{thmtools}
\usepackage{graphicx}
\usepackage{xcolor}
\usepackage{latexsym}
\usepackage{etoolbox}
\usepackage{accents}
\usepackage{stackrel}
\usepackage{epstopdf}
\usepackage[shortlabels]{enumitem}
\usepackage{lipsum}
\usepackage[hidelinks]{hyperref}
\usepackage[nameinlink,capitalize]{cleveref}
\usepackage[thumblink=none,linefill=dots,height={50pt},minheight={33pt},%
            width={128pt},distance={2mm},topthumbmargin={66pt},bottomthumbmargin={66pt},%
            nophantomsection=false,ignorehoffset=true,ignorevoffset=true,final=true,%
            hidethumbs=false,verbose=true]{thumbs}
\usepackage{tikz}
\usepackage{tikz-cd}
\usepackage[USenglish]{isodate}

\numberwithin{equation}{section}

\makeatletter
\ifcsname phantomsection\endcsname
    \newcommand*{\qrr@gobblenexttocentry}[5]{}
\else
    \newcommand*{\qrr@gobblenexttocentry}[4]{}
\fi
\newcommand*{\addsubsection}{%
    \addtocontents{toc}{\protect\qrr@gobblenexttocentry}%
    \subsection}
\makeatother

\declaretheorem[name=Theorem, numberwithin=section]{theorem}
\declaretheorem[name=Theorem, numbered=no]{theorem*}

\declaretheorem[name=Lemma, sibling=theorem]{lemma}
\declaretheorem[name=Proposition, sibling=theorem]{proposition}
\declaretheorem[name=Corollary, sibling=theorem]{corollary}

\declaretheorem[style=definition, name=Definition, sibling=theorem]{definition}
\declaretheorem[style=definition, name=Remark, sibling=theorem]{remark}
\declaretheorem[style=definition, name=Example, sibling=theorem]{example}

\makeatletter
\newcommand{\customlabel}[2]{%
   \protected@write \@auxout {}{\string \newlabel {#1}{{#2}{\thepage}{#2}{#1}{}} }%
   \hypertarget{#1}{}%
}
\makeatother

\title{{\bf Univalent Foundations of Constructive Algebraic Geometry}}
\author{Max Zeuner}
\date{{\footnotesize \sc Department of Mathematics, Stockholm University} \\ {\footnotesize \href{mailto:zeuner@math.su.se}{\texttt{zeuner@math.su.se}}} \\[1.5em] \today}

\usepackage{relsize}
\usepackage{nicefrac}
\usepackage{graphicx}
\usepackage{array}
\usepackage{quiver}
\usepackage{xspace}
\usepackage{bm}
\usepackage{stmaryrd}

\newcommand{\N}{\mathbb{N}}
\newcommand{\Z}{\mathbb{Z}}
\newcommand{\ZL}{\ensuremath{\mathcal{L}_R}}

\newcommand{\NatTrans}[2]{{#1}\Rightarrow{#2}}
\newcommand{\coBrackets}[1]{\llbracket\,{#1}\,\rrbracket}

\newcommand{\catHom}[3]{#1\big(#2,#3\big)}
\newcommand{\locEl}[2]{#1[\nicefrac{1}{#2}]}
\newcommand{\ZarLat}[1]{\mathcal{L}_{#1}}
\newcommand{\Support}[2]{\mathcal{D}_{#1}(#2)}
\newcommand{\rest}[2]{#1\!\!\restriction_{#2}}
\newcommand{\Aone}{\mathbb{A}^1}
\newcommand{\Ofun}{\mathcal{O}}
\newcommand{\ZFunctor}{\mathbb{Z}\text{-}\mathsf{Fun}}
\newcommand{\fpAlg}[1]{#1\text{-}\mathsf{Alg}_{fp}}

\newcommand{\vcong}{\mathbin{\text{\rotatebox[origin=c]{90}{$\simeq$}}}}
\newcommand{\vvcong}{\mathbin{\text{\rotatebox[origin=c]{-90}{$\simeq$}}}}

\DeclareMathOperator{\Spec}{\mathsf{Spec}}

\newcommand{\tySigmaNoParen}[3]{{\{\,#1:#2~\vert~#3\,\}}}

\newcommand{\tyPath}[2]{{#1}={#2}}

\newcommand{\pbsign}[2][2]{\ar[#2,phantom,pos=0,"\scalebox{#1}{\tiny{$\lrcorner$}}"]}

\DeclarePairedDelimiter\abs{\lvert}{\rvert}
\DeclarePairedDelimiterX{\norm}[1]{\lVert}{\rVert}{#1}

\begin{document}
\pagenumbering{arabic}
\maketitle

\begin{abstract}
\noindent
We investigate two constructive approaches to defining quasi-compact
and quasi-separated schemes (qcqs-schemes), namely qcqs-schemes as
locally ringed lattices and as functors from rings to sets. We work in
Homotopy Type Theory and Univalent Foundations, but reason informally.
The main result is a constructive and univalent proof that the two
definitions coincide, giving an equivalence between the respective
categories of qcqs-schemes.
\end{abstract}

\paragraph*{Acknowledgments} The author would like to thank Thierry Coquand
for suggesting this project and for his invaluable input and feedback.
Further thanks go to Anders Mörtberg for his supervision and to
Felix Cherubini and Matthias Hutzler for being an integral part of the project.

\section{Introduction}

Since its inception by Grothendieck and his school in the latter half
of the 20th century, the notion of scheme has found important
applications in various fields of mathematics ranging from geometry to
number theory.  Arguably, it has led to some of the most staggering
mathematical achievements in recent history. In spite of being a rather
abstract notion that relies heavily on modern set-theoretic machinery,
the motivations for studying schemes coming from classic algebraic
geometry have a rather hands-on computational character: studying the
solutions to finite systems of polynomial equations.  Giving a
satisfying constructive definition of schemes and determining the
computational content of scheme theory, is thus an important problem
for the program of constructive mathematics.

In this task, the constructive mathematician is confronted with two
options: Constructivizing the topological notion of scheme as a
locally ringed space or the categorical notion of scheme as a
``functor of points''.  At first glance, categorical/functorial
schemes seem more appropriate for constructive study, as topological
spaces are an inherently problematic notion from a constructive point
of view.  The categorical approach was actually the preferred one of Grothendieck
himself \cite{GrothendieckBuffalo}. Furthermore, all of the basic notions of the
functorial approach, as e.g.\ defined in the textbook
``Introduction to algebraic geometry and algebraic groups''
By Demazure and Gabriel \cite{DemazureGabriel}, are constructively valid.  Moreover,
the functorial schemes embed into a sheaf topos, the so-called
\emph{big Zariski topos} and in the internal language of this topos, which is
constructive, one can develop algebraic geometry synthetically. This
line of work was pioneered by Kock \cite{Kock76} and is studied
extensively in the thesis of Blechschmidt \cite{BlechschmidtPhD}.
The synthetic approach has also been studied using Homotopy Type Theory
and Univalent Foundations by Cherubini, Coquand and Hutzler \cite{SAG}.

As its name suggests, however, the big Zariski topos cannot be defined
rigorously without introducing size issues, which arise
already when working classically. Demazure and Gabriel
assume \emph{two} Grothendieck universes for precisely this reason
\cite{DemazureGabriel}. More modern sources usually tend to ignore
size issues or leave ``the appropriate modifications'' to
resolve these issues ``to the interested reader'' \cite{Jantzen1987}.
Of course, the size issues become even more pressing in a constructive,
predicative setting.  Blechschmidt's thesis contains an insightful
discussion \cite[Sec.\ 16.5]{BlechschmidtPhD}, highlighting that it is
possible to work in an entirely constructive and predicative set-up
for synthetic algebraic geometry, by restricting oneself to
\emph{finitely presented} algebras over some base ring $R$. One can
then define \emph{schemes of finite presentation over}
$\mathsf{Spec}(R)$ in a purely functorial way that corresponds to a
well-behaved internal definition of ``finitely presented synthetic schemes''
\cite[Def.\ 19.49]{BlechschmidtPhD}.

In this context it is noteworthy that all \emph{quasi-compact and
quasi separated} schemes (qcqs-schemes) can be defined constructively
and predicatively. These qcqs-schemes contain all schemes of finite
presentation and need not be defined over some affine base scheme.
This is achieved through a careful point-free reformulation of the
topological approach. Initiated in a series of papers by Schuster
\cite{SchusterHabil,SchusterPositivity,SchusterZariski} using formal
topology, this approach was subsequently refined by Coquand, Lombardi
and Schuster \cite{ProjSpec,ConstrSchemes} using lattices.  The main
observation is that, since qcqs-schemes are
always a coherent or spectral space, it is sufficient to consider the
distributive lattice of quasi-compact opens and the structure sheaf
restricted to this lattice. Instead of locally ringed spaces,
qcqs-schemes are defined as ringed lattices in \cite{ConstrSchemes}.

In this paper we want to establish a common constructive framework
for the functorial and the lattice-theoretic approach.
Classically, to each locally ringed space $X$ one can associate a sheaf
in the big Zariski topos called the functor of points of $X$.
This is the construction that most textbooks use, when developing
the functorial approach from the standard classical approach.\footnote{
  See e.g.\ \cite{EisenbudHarris}.}
Modulo size issues,
this functor of points construction has a left adjoint
and when restricted schemes this adjunction becomes an
equivalence of categories. This result is called the
\emph{comparison theorem} in \cite[Ch.\ I, \S 1, no 4.4]{DemazureGabriel}
and the main result of this paper is a constructive
version of this result for qcqs-schemes and schemes of finite presentation.

The main obstacle are not the size issues, but the
circumstance that qcqs-schemes in \cite{ConstrSchemes} are defined as
ringed lattices, while the functor of points construction requires a
\emph{locally} ringed space.  What is needed are notions of
\emph{locally ringed lattices} and their \emph{morphisms}, such that a
morphism of qcqs-schemes (in the sense of \cite{ConstrSchemes}) is just
a morphism of locally ringed lattices.  In \cite{ConstrSchemes}, the
study of such notions is omitted on purpose.
The authors argue that the paper shows that
\begin{quote}
``\emph{up to a certain point one can get by on without any
talk of locally ringed lattices and their morphisms, concepts which seem relatively
involved if compared with the ones given in this paper.}''
\end{quote}
Yet, they do conclude that the study of locally ringed lattices is
central to a fully constructive foundation of algebraic geometry.  In
this paper we want to fill this gap and provide an in-depth discussion
of locally ringed lattices and their morphisms.  This is not just a
means to the end of obtaining a constructive comparison theorem. We
want to show that qcqs-schemes as locally ringed lattices actually admit a rather neat
presentation.

The basics of the constructive approach to functorial qcqs-schemes
that we use in this paper have actually been formalized in the
\texttt{Cubical Agda} \cite{ZeunerHutzler24}. The
size issues are tackled by using enough type-theoretic universes. As
\texttt{Cubical Agda} is an extension of  the \texttt{Agda} proof assistant supporting
Homotopy Type Theory and Univalent Foundations (HoTT/UF), these
universes are univalent.  We use HoTT/UF as the foundation for this
work as well and make crucial use of univalence.  This choice
of foundation for (constructive) algebraic geometry is
further motivated by a \texttt{Cubical Agda} formalization
of the structure sheaf of an affine scheme \cite{ZeunerMortberg23},
suggesting that the structuralist approach to
mathematics of the Grothendieck school can be emulated more faithfully using
univalence.  Generally, HoTT/UF is not only a foundation for
``higher'' mathematics and synthetic homotopy theory. The internal
notions of sets and propositions provide a natural framework for
set-level constructive mathematics with convenient features like
quotients. In the particular case of the constructive comparison
theorem, univalent type-theoretic universes let us circumvent size issues
and actually strengthen the main result.

This paper tries to strike a delicate balance between conceptual
clarity and formal rigor. On the one hand, many of the insights
presented in the paper do not actually depend on the choice of
foundation and we do not want to bury them under technicalities. On
the other hand, with important parts of constructive scheme theory
already formalized in the univalent \texttt{Cubical Agda} proof
assistant \cite{ZeunerMortberg23,ZeunerHutzler24}, this paper aims to
serve as a sort of blue-print for formalized univalent and
constructive algebraic geometry in the tradition of Voevodsky's
Foundations library \cite{VoevodskyFoundationsLib}.
The paper is structured as follows:

\begin{itemize}
\item In \cref{sec:Background} we give some of the necessary background.
  We comment on our use of HoTT/UF and provide some guidance for
  readers not familiar with type theory or univalence. We also
  introduce the so-called Zariski lattice of a commutative ring. This construction
  has been studied extensively in point-free topology and constructive
  algebra and will play an important role throughout this paper.
\item In \cref{sec:LRDL} we introduce the notion of a locally ringed lattice,
  building on an idea of Coquand.
  We show that the Zariski lattice of a commutative ring $R$ together with
  its structure sheaf is a locally ringed lattice, which can be regarded
  as the constructive spectrum of $R$. We give a simple proof that the
  thus induced spectrum functor is adjoint to the global sections functor.
  We define qcqs-schemes as locally ringed lattices and show that
  scheme morphisms as defined in \cite{ConstrSchemes} are just
  morphisms of locally ringed lattices. We finish the section discussing
  why schemes defined as locally ringed lattices coincide with
  classical qcqs-schemes.
\item In \cref{sec:fopApproach} we describe the functor of points approach
  to qcqs-schemes. We build on the \texttt{Cubical Agda} formalization
  described in \cite{ZeunerHutzler24}, but provide strengthened results.
  We define $\Z$-functors, i.e.\ functors from rings to sets, and the
  Zariski coverage on $\Z$-functors. We introduce compact open subfunctors
  as subobjects of $\Z$-functors that are classified by the Zariski lattice and
  use these to define functorial qcqs-schemes. We prove that compact open
  subfunctors of functorial qcqs-schemes are qcqs-schemes.
\item In \cref{sec:ComparisonThm} we prove a ``comparison theorem''
  for the two notions of qcqs-schemes given in this paper. The proof
  can be seen as an adaption of the comparison theorem of Demazure and
  Gabriel to HoTT/UF,
  making some of the steps more explicit and precise.
  We define the functor of points of a locally ringed lattice and
  the realization of a $\Z$-functor as a locally ringed lattice.
  We show that these functors are adjoint modulo size issues and that
  on (functorial) qcqs-schemes both functors become fully faithful.
  With the help of univalence, we can ignore size issues and obtain
  an equivalence of the respective categories of qcqs-schemes.
\item In \cref{sec:SchemesOfFP} we give an outline of how to adapt the above results to
  schemes of finite presentation over some base ring $R$, without
  having to deal with size issues.  The classical definition is
  adapted to locally ringed lattices and for functorial schemes of finite presentation
  we use Blechschmidt's parsimonious site of
  finitely presented $R$-functors \cite[Sec.\ 16.5]{BlechschmidtPhD}.
  We indicate why in this setting
  the proof of the comparison theorem not only still goes through,
  but can actually be simplified.
\end{itemize}

\section{Background}\label{sec:Background}

This paper aims to give a self-contained introduction to the
constructive theory of schemes from the univalent point of
view. Familiarity with classical algebraic geometry and schemes as
e.g.\ introduced in Hartshorne's classic ``Algebraic Geometry''
\cite{Hartshorne} is not strictly required, although certainly helpful
to gain an intuition for the notions introduced in this paper and for
understanding the motivations to study these notions in the first
place.  An exception is \cref{subsec:classicalChar}, where we compare
our definition using locally ringed lattices to the classical one.
What is assumed as a prerequisite, however, is some commutative
algebra including the basic theory of ideals
and localizations of commutative rings, i.e.\ rings of fractions.\footnote{
  These prerequisites are covered e.g.\ in the first three chapters
  of ``Introduction to Commutative Algebra'' by Atiyah and MacDonald \cite{AtiyahMacDonald}.}
On top of that we will also assume knowledge of category theory,
in particular familiarity with (co-) limits, adjunctions, presheaves,
the Yoneda lemma and Kan extensions.\footnote{
  This is covered in e.g.\ chapters I-V and X of Mac Lane's
  ``Categories for the working mathematician'' \cite{MacLaneCategories}
  or Riehl's ``Category Theory in Context'' \cite{RiehlCatTheory}.}

\subsection{Univalent foundations} %%%%%%%%%%%%%%%%%%%%%%%%%%%%%%%%%%%%%%%%%%%%%%%%%%%%%%%%%%%%%%%%%%

We work in HoTT/UF and all definitions and results are
written in the informal (or rather semi-formal) style of the ``HoTT-book''
\cite{HoTTBook}. The paper aims to present the key
definitions and ideas with only a modicum of type-theoretic notation
that we quickly want to recapitulate here. For understanding the
details of this paper and the relation to the
\texttt{Cubical Agda} formalizations \cite{ZeunerMortberg23,ZeunerHutzler24},
acquaintance with HoTT/UF is required.\footnote{
  This roughly amounts to the material covered in first three chapters
  of the HoTT-book together with some additional topics like set-quotients,
  truncations and of course the consequences of univalence.}
The notation largely follows the conventions of the HoTT-book
\cite{HoTTBook}, with some exceptions inspired by the syntax of the
\texttt{Agda} proof assistant. We are first and foremost concerned
with types and their terms rather than sets and their elements and thus
write $x:A$ to denote that $x$ is of type $A$, instead of $x\in A$.
Note that this is a stipulation, like introducing a new variable,
rather than a provable statement.
In particular, there is no general type-theoretic equivalent to $x\notin A$.

HoTT/UF is most commonly described as an extension of Martin-Löf's
\emph{dependent type theory} (MLTT) \cite{MartinLof75itt}.
For our purposes, we actually only need rather few additions
to MLTT namely, the univalence axiom and set-quotients,
which can then be used to define propositional truncations.
For a type $A$, a dependent type $B$ over $A$
can be seen as a family of types $B(a)$ for $a:A$.
The type of \emph{dependent functions} $(a:A)\to B(a)$ is the type of
functions $f$ that take an input $a:A$ and return an output $f(a):B(a)$,
i.e.\ functions that vary in their codomain depending on the input.
Note that we are using \texttt{Agda}-syntax for dependent functions
rather than $\Pi$-types in order to avoid confusion with products.
The $\Sigma$-type or type of \emph{dependent pairs} $\Sigma_{a:A}B(a)$ is
the type of pairs $(a,b)$ such that $a:A$ and $b:B(a)$, i.e.\
pairs where the type of the second component depends on the first component.
These are incredible useful to define complex structures. However,
we will mostly introduce new structures informally and avoid
writing $\Sigma$-types in order to increase readability.

Martin-Löf type theory has two notions of equality. Definitional
equality, which we will denote by $\equiv$, is a stipulated equality
that cannot be (dis-) proven.  We follow the HoTT-book and write
$:\equiv$ for definitional equalities that introduce new objects or
notation.  On top of that there are \emph{identity types}. For
$x,y:A$, we may think of the type $x=y$ as the type of proofs that $x$
equals $y$.  Identity types can be rather complex objects in HoTT/UF,
but they let us define the notions of set and proposition internally,
meaning that we can define inside HoTT/UF, which types are
propositions, sets or neither. Propositions are subsingleton types or
types with at most one element. In other words, a proposition is a
type $P$ such that for all $x,y:P$ we have $x=y$. The intuition for
this is that giving an element of $P$ amounts to proving $P$.
A proposition can be given a proof, but no further structure.
There is nothing more to a proof than
establishing the truth of the corresponding proposition,
as any two proofs are regarded as equal.

A type $S$ is called a set if for any $x,y:S$, the type of proofs of equality
$x=y$ is a proposition. The intuition for this definition is rooted in
the discrete nature of sets, where elements can be equal or not,
but do not share any other relation. A subset of $S$ is a function
$S\to\mathsf{Prop}$ from $S$ into the type of propositions. This
should be taken with a grain of salt, as we need to introduce
universes to talk about things like the type of all propositions
without running into paradoxes, see the discussion below. Only for a
subset $X$ of $S$, will we us the notation $x\in X$, denoting the type
of proofs that $x:S$ is in $X$, i.e.\ the proposition that $X$ (seen
as a function) maps $x$ to.  For the $\Sigma$-type of $x:S$
with a proof that $x\in X$, we will use set-theory inspired notation writing
$\{\,x:S\,\vert\,x\in X\,\}$ as in the HoTT-book
\cite[3.5.2]{HoTTBook}.  As the proof of $x\in X$ can often be
ignored, we will sometimes write $x:X$ for introducing and element of
$S$ that belongs to $X$.  Strictly speaking, this is an abuse
notation, identifying $X$ as a function with the above $\Sigma$-type,
but it reduces verbosity.  The univalence axiom actually implies that
certain types are not sets.  For sets $X$ and $Y$, the type $X=Y$ is
not a proposition, but a set and it is in bijection with the set of
bijections from $X$ to $Y$.  This means that the category of sets is a
univalent category and we will discuss below, why this is actually a
helpful feature for our purposes.

In order to describe existential quantification as a proposition,
we need so-called \emph{propositional truncation}, an operation that turns
any type into a proposition, by ``truncating'' or ``squashing'' any structure
on objects. Following the HoTT-book we say that there \emph{merely exists} an $a:A$
such that $P(a)$ if we have an object/proof of the truncated
$\Sigma$-type, which we will write as $\exists_{a:A}P(a)$.
Note that the truncation prevents us from extracting the ``witness'' $a$.
This way choice principles when stated in terms of mere existence
become unprovable in HoTT/UF. This is of course in line with
constructive mathematics. Reasoning constructively about sets
using propositions and truncations in HoTT/UF requires some getting used
to, especially when concepts like set-quotients are involved.\footnote{
  Here, we will not get into the details of how set-quotients
  are defined, but obersve that for sets they work very much as expected.}
Ultimately, it is a rather rewarding approach
that can provide some useful intuition for constructivity issues.

A point in which we depart from the HoTT-book in more than just
notation is the treatment of universes. We do a assume a (potentially
infinite) hierarchy of univalent universes, closed under the usual type
formers, indexed by their \emph{level}: $\mathsf{Type}_0,\mathsf{Type}_1,$ etc.
However, we do not take these universes to be cumulative.
Instead, we assume explicit embeddings
$\mathsf{lift}:\mathsf{Type}_\ell\hookrightarrow\mathsf{Type}_{\ell+1}$
commuting with all the type formers. The reason for this is that we
want our work to be in line with the results of
\cite{ZeunerMortberg23,ZeunerHutzler24} that are already formalized in
\texttt{Cubical Agda}, which does not have cumulative universes.
To remind the reader that we work with non-cumulative universes, we use
\texttt{Agda}-like notation, but we omit levels whenever we are working within a
single universe. For example, we write $\mathsf{Set}$ for the type or category
of sets in a universe $\mathsf{Type}$ at an unspecified level. This type itself
lives in the successor universe, but this can often be ignored.
If size issues do matter, we will always be careful to annotate universe levels.
Note that we are not assuming any form of impredicativity for our universes,
such as Voevodsky's resizing axioms \cite{VoevodskyResizing}. This means that the
type $\mathsf{Prop}_\ell$ of propositions in $\mathsf{Type}_\ell$  lives
in $\mathsf{Type}_{\ell+1}$. Similarly, for a set $S:\mathsf{Type}_{\ell}$ the
type of subsets $S\to\mathsf{Prop}_\ell$ lives in $\mathsf{Type}_{\ell+1}$ as well.

Apart from the basics, we use concepts and results from category
theory in HoTT/UF, as described in chapter 9 of the HoTT-book \cite{HoTTBook}. An
important difference in nomenclature is that by a \emph{category} we
denote what the HoTT-book would call a pre-category
\cite[Def.\ 9.1.1]{HoTTBook}, i.e.\ a structure consisting of a
\emph{type of objects} and \emph{sets of arrows} between any two
objects with identities and composition, satisfying the usual laws.
For a category $\mathcal{C}$, we denote (with some abuse of notation)
its type of objects by $\mathcal{C}$ as well and the type of arrows as
$\catHom{\mathcal{C}}{x}{y}$ for $x,y:\mathcal{C}$.
For example $\mathsf{Set}$ is both the type of sets and
the corresponding category (with functions as arrows) in a given universe.
The basic theory of functors, natural transformations, (co-) limits
and the Yoneda lemma can be developed in a very familiar,
straightforward fashion in HoTT/UF for this notion of category.  We
write functors as if they were functions $F:\mathcal{C}\to\mathcal{D}$
and omit projections for the action on objects or arrows.  In
particular, we have $F(x):\mathcal{D}$ for $x:\mathcal{C}$ and
$F(f):\catHom{\mathcal{D}}{F(x)}{F(y)}$ for
$f:\catHom{\mathcal{C}}{x}{y}$.  Natural transformations between
functors $F$ and $G$ are denoted by $\NatTrans{F}{G}$ and
the category of functors from
$\mathcal{C}$ to $\mathcal{D}$ is denoted by
$\catHom{\mathsf{Fun}}{\mathcal{C}}{\mathcal{D}}$.

Of special interest are so-called \emph{univalent categories}. Those
are categories for which the identity types on objects are canonically
equivalent to types of isomorphisms (defined through inverse arrows).
Arguably, they are the central notion of univalent category theory
and are thus simply called categories in the HoTT-book \cite[Def.\ 9.1.6]{HoTTBook}.
They are, however, a truly novel concept that has no counterpart outside of HoTT/UF.
Certain things that are usually not provable in constructive category theory
do hold for univalent categories in HoTT/UF. If $F:\mathcal{C}\to\mathcal{D}$
is a functor between univalent categories that is fully-faithful and essentially surjective,
we can actually construct an adjoint inverse to $F$
\cite[Lemma 9.4.5 \& 9.4.7]{HoTTBook}.
This will be used in the main \cref{thm:comparisonThm}.

Proving that a category is univalent usually requires the univalence
axiom, or rather the so-called ``structure identity principle'' in combination
with univalence \cite[Sec.\ 9.8]{HoTTBook}. This way, one can
prove that not only the category of sets $\mathsf{Set}$, but also
categories of algebraic structures like commutative rings $\mathsf{CommRing}$ and
distributive lattices $\mathsf{DistLattice}$ are univalent. There are several
versions of the structure identity principle in the literature,\footnote{See \cite{IsoIsEq,HoTTBook,Escardo22,DisplayedCats,DURGs,POPLPaper}.}
so we do not want to go into details here. However, we will
use it to prove that locally ringed lattices form a univalent category in
\cref{prop:isUnivalentLRDL}. Readers not familiar with univalent category
theory can regard this fact together with the aforementioned fact about
functors between univalent categories as a black box,
only required in the last step of proving the comparison theorem.

When working with univalent categories we can only give definitions and proofs
that are invariant under isomorphisms. As mentioned in the introduction,
this seems to fit the approach of the Grothendieck school. It is then not too
surprising that virtually all of the categories appearing in this paper are univalent.
As we saw for sets, however,
objects of a univalent category do not necessarily form a set.
This is not an issue of size but of identity types. The type of objects of a univalent
category is in general a so-called ``homotopy groupoid''. Even though the aim of this paper is
to develop set-level constructive mathematics in HoTT/UF, non-sets will thus
play an important role. The only place in the paper, where this is
actually an issue is \cref{ex:affSchLRDL}, where we define the
structure sheaf of an affine scheme. Here, we do not want to get into
any details and refer the interested reader to \cite{ZeunerMortberg23}.
Especially in view of the results of \cite{ZeunerMortberg23} and this paper,
we would like to argue that
one should regard the occurrence of non-sets
as a feature rather than a bug of univalent foundations of algebraic geometry.

\subsection{The Zariski lattice} %%%%%%%%%%%%%%%%%%%%%%%%%%%%%%%%%%%%%%%%%%%%%%%%%%%%%%%%%%%%%%%%%%%%

Both definitions of qcqs-schemes that we will give in this paper
crucially rely on an important notion from constructive algebra, called
the Zariski lattice of a ring. It can be seen as the constructive
counterpart to the Zariski spectrum of a ring. A common starting point for
algebraic geometry is the observation that for a
any ring $R$, the set of prime ideals $\mathfrak{p}\subseteq R$ can
be equipped with a topology, the so-called Zariski topology. One can then
define affine schemes and schemes using this ``spectrum''
$\mathsf{Spec}(R)$.  The Zariski open sets are generated by basic
opens. For $f:R$, the corresponding basic open is given as
\begin{align*}
  D(f)=\big\{\mathfrak{p}~\vert~f\notin\mathfrak{p}~\big\}
\end{align*}
From a constructive point of view this definition is not really workable.
Point-set topology is already in itself inherently non-constructive and in the
case of the Zariski spectrum of some general $R$, it is impossible to prove that there are any points
(prime ideals) without invoking Zorn's lemma or some other choice principle.
One can, however, restrict to (quasi-) compact open subsets, which form a distributive lattice
with finite set-theoretic union and intersection, the so-called Zariski lattice
$\ZL$.
Joyal observed that $\mathcal{L}_R$ is described by a universal property,
which is stated in terms of so-called \emph{supports} \cite{JoyalZarLat}.
Note that in this paper we will always work with bounded distributive lattices,
i.e.\ distributive lattices with a top and bottom element,
and commutative rings with $1$.
\begin{definition}
  Let $R$ be a commutative ring and $L$ be a distributive lattice.
  A map $d:R\to L$ is called a \emph{support} of $L$ if
  \begin{enumerate}
  \item $d(0)=0$ and $d(1)=1$
  \item $d(xy)=d(x)\wedge d(y)$ for all $x,y: R$
  \item $d(x+y)\leq d(x) \vee d(y)$ for all $x,y: R$
  \end{enumerate}
\end{definition}

Basic opens are always compact and the map $D:R\to\ZL$ is a support.
In fact, it is a universal support in the following sense:
For any support $d:R\to L$ there is a unique lattice homomorphism
such that the following diagram commutes
\[
\begin{tikzcd}
  & R \arrow[dl,"D"']\arrow[dr,"d"] & \\
  \mathcal{L}_R \arrow[rr,dashed, "\exists!"'] && L
\end{tikzcd}
\]
This means that we can define $\ZL$ as the lattice generated by formal elements $D(f)$,
modulo the equations given by the support conditions.
There is also a more direct but still point-free way to construct the Zariski lattice.
This is the approach taken by Espa\~{n}ol \cite{Espanol83}
and Coquand and Lombardi \cite{CoquandLombardiKrullDim}.\footnote{The term Zariski lattice
  was coined  in the paper by Coquand and Lombardi.}
It allows one to prove another important property of $\ZL$.
The formalization of the Zariski lattice described in
\cite{ZeunerMortberg23} adopts this approach to the setting of \texttt{Cubical Agda},
i.e.\ to a dependent type theory with set-quotients but without impredicativity
in the form of propositional resizing.

Classically, the Zariski open sets of $\mathsf{Spec}(R)$ are in bijection with the
radical ideals of $R$. Recall that an ideal $I\subseteq R$ is called
a radical ideal if $I=\sqrt{I}$, where
$\sqrt{I}:\equiv\{\, x: R \;\vert\; \exists_{n:\N}~x^n\in I\,\}$.
Predicatively, the type of ideals of $R:\mathsf{CommRing}_\ell$
lives in $\mathsf{Type}_{\ell+1}$. So, even though this gives us a point-free
description of Zariski opens, it also introduces size issues or requires us
to assume resizing. Fortunately, the situation is better behaved for (quasi-) compact opens.
These are in bijection with radicals of finitely generated ideals, i.e.\
ideals of the form $\sqrt{\langle f_1,...,f_n\rangle}$, where
$f_1,...,f_n:R$. Moreover, union and intersection of (quasi-) compact opens,
i.e.\ join and meet of $\ZL$, correspond to addition and multiplication
of finitely generated ideals. In other words, we can define $\ZL$ as follows.

\begin{definition}
  The \emph{Zariski lattice} $\ZL$ of a ring $R$ has as elements
  lists of generators $f_1,...,f_n:R$
  modulo taking the radical of the ideal generated by the $f_i$.
  This means that lists of generators $f_1,...,f_n:R$ and $g_1,...,g_m:R$
  are in the same equivalence class of $\ZL$, if
  $\tyPath{\sqrt{\langle f_1,\dots,f_n\rangle}}{\sqrt{\langle g_1,\dots,g_m\rangle}}$.\footnote{
    Technically, the identity type $I=J$ of two (radical) ideals lives in $\mathsf{Type}_{\ell+1}$,
    but in this particular instance we can replace it with a small proposition, as
    $\tyPath{\sqrt{\langle f_1,\dots,f_n\rangle}}{\sqrt{\langle g_1,\dots,g_m\rangle}}$
    is equivalent to all the $f_i$ being in ${\sqrt{\langle g_1,\dots,g_m\rangle}}$ and vice versa.}
  The equivalence class of the generators
  $f_1,...,f_n:R$ is denoted by $D(f_1,...,f_n):\ZL$. Join and meet are given by
  \begin{align*}
    D(f_1,...,f_n)\vee D(g_1,...,g_m)&:\equiv D(f_1,...,f_n,g_1,...,g_m) \\
    D(f_1,...,f_n)\wedge D(g_1,...,g_m)&:\equiv D(f_1g_1,...,f_ig_j,...,f_ng_m)
  \end{align*}
\end{definition}

The map $D:R\to\ZL$, sending $f:R$ to $D(f)$, the equivalence class of the radical
of the principal ideal generated by $f$, $\sqrt{\langle f\rangle}$, is a support.
It is straightforward to show that $\ZL$ and $D$ thus defined, satisfy the universal
property of the Zariski lattice. However, we can also prove the following characterization
of the induced order on $\ZL$:
\begin{align*}
  D(f_1,...,f_n)\leq D(g_1,...,g_m) ~&\Leftrightarrow~ {\sqrt{\langle f_1,\dots,f_n\rangle}}\subseteq{\sqrt{\langle g_1,\dots,g_m\rangle}} \\
  &\Leftrightarrow~ f_i\in{\sqrt{\langle g_1,\dots,g_m\rangle}}~\text{for}~i=1,...,n
\end{align*}
This fact\footnote{This fact is called the ``formal  Hilbert Nullstellensatz''
  in \cite{ConstrSchemes}.}
and the universal property let us prove all the relevant properties of the
Zariski lattice without having to appeal to prime ideals.

If $\varphi:\mathsf{Hom}(A,B)$ is a morphism of rings, $D\circ\varphi:A\to\ZarLat{B}$ is a
support and we get an induced morphism between Zariski lattices
\[
\begin{tikzcd}
  & A \arrow[dl,"D"']\arrow[dr,"D\circ\varphi"] & \\
  \ZarLat{A} \arrow[rr,dashed, "\exists!"'] && \ZarLat{B}
\end{tikzcd}
\]
which we will denote by $\varphi^{\mathcal{L}}$. Classically, this lattice morphism
takes pre-images of (quasi-) compact opens under the continuous map
$\mathsf{Spec}(B)\to\mathsf{Spec}(A)$ induced by $\varphi:\mathsf{Hom}(A,B)$.
For $f_1,...,f_n:R$, we have
$\varphi^{\mathcal{L}}\big(D(f_1,\dots,f_n)\big)=D(\varphi(f_1),\dots,\varphi(f_n))$.

An important fact in classical algebraic geometry is that $D(f)$ (as
a subspace of $\mathsf{Spec}(R)$) is homeomorphic to $\mathsf{Spec}(\locEl{R}{f})$.
The result carries over to the Zariski lattice,
giving a bijection between (quasi-) compact opens of the two spaces.
We want to finish this section on the Zariski lattice by
giving a point-free proof of this fact that only uses the universal property.
The result will be of importance later. For the remainder of this section
we fix an element $f:R$. Recall that the ring $\locEl{R}{f}$ is a special
case of a localization of $R$ \cite[Ch.\ 3]{AtiyahMacDonald}.

\begin{definition}
  The \emph{localization of} $R$ \emph{away from} $f$ is the ring
  $\locEl{R}{f}$ of fractions $\nicefrac{r}{f^n}$ where
  $r:R$ and the denominator is a power of $f$.
  Equality of two fractions is given by: $\tyPath{\nicefrac{r}{f^n}}{\nicefrac{r'}{f^m}}$
  if and only if there merely exists a $k\geq 0$ such that $\tyPath{rf^{k+m}}{r'f^{k+n}}$.
  The universal property of the localization away from $f$ can be stated as:
  For any ring $A$ with a homomorphism $\varphi:\mathsf{Hom}(R,A)$ such that
  $\varphi(f)\in A^\times$ (i.e.\ $\varphi(f)$ is a unit/invertible),
  there is a unique $\psi:\mathsf{Hom}(\locEl{R}{f},A)$ making
  the following diagram commute
  \[
  \begin{tikzcd}
    &R\arrow[ld,"\nicefrac{\_}{1}"']\arrow[rd,"\varphi"]&\\
    \locEl{R}{f}\arrow[rr,dashed,"\exists!~\psi",swap]&&A
  \end{tikzcd}
  \]
  where $\nicefrac{\_}{1}:\mathsf{Hom}(R,\locEl{R}{f})$ is the canonical morphism
  mapping $r:R$ to the fraction $\nicefrac{r}{1}$. In other words, $\locEl{R}{f}$
  is the initial $R$-algebra where $f$ becomes invertible.
\end{definition}

\begin{lemma}\label{lem:isoLocDownSet}
  For $f:R$ we have an isomorphism of lattices
  $\psi_f:\ZarLat{\locEl{R}{f}} ~\cong~ \downarrow\! D(f)$
  between the Zariski lattice of the localization $\locEl{R}{f}$
  and the elements of $\ZL$ below $D(f)$, i.e\
  the down-set $\downarrow\! D(f):\equiv~\tySigmaNoParen{u}{\ZarLat{R}}{u\leq D(f)}$.
\end{lemma}

\begin{proof}
  Consider the map $d:\locEl{R}{f}\to\,{\downarrow\!{D(f)}}$ given by
  $d(\nicefrac{r}{f^n}):\equiv~ D(r)\wedge D(f)$. This defines a support and thus induces
  $\psi_f:\ZarLat{\locEl{R}{f}}\to\,{\downarrow\! D(f)}$. For the inverse direction
  take the support $D(\nicefrac{\_}{1}):R\to\ZarLat{\locEl{R}{f}}$
  and the induced map $(\nicefrac{\_}{1})^{\mathcal{L}}:\ZL\to\ZarLat{\locEl{R}{f}}$.
  Restricting this to $\downarrow\! D(f)$,
  gives us the inverse $\psi_f^{-1}$.
  Note that this is indeed a lattice morphism, since
  the top element of $\downarrow\! D(f)$ is $D(f)$
  and $\psi_f^{-1}(D(f))=D(\nicefrac{f}{1})=D(\nicefrac{1}{1})$
  as $\nicefrac{f}{1}$ is a unit in $\locEl{R}{f}$.

  In order to prove that the two maps are inverse to each other,
  we claim that the map $\_\wedge D(f):\ZL\to\,{\downarrow\! D(f)}$
  factors through $\psi_f$. In particular, we claim that the following diagram commutes
  \[\begin{tikzcd}
          R && {\locEl{R}{f}} \\
          \ZL && {\ZarLat{\locEl{R}{f}}} && {\downarrow D(f)}
          \arrow["{d}", from=1-3, to=2-5]
          \arrow["(\nicefrac{\_}{1})^{\mathcal{L}}", from=2-1, to=2-3]
          \arrow["{\psi_f}", from=2-3, to=2-5]
          \arrow["{\_\wedge D(f)}"', curve={height=18pt}, from=2-1, to=2-5]
          \arrow["D"', from=1-3, to=2-3]
          \arrow["{\nicefrac{\_}{1}}", from=1-1, to=1-3]
          \arrow["D"', from=1-1, to=2-1]
  \end{tikzcd}\]
  This follows from the universal property of $\ZL$,
  which gives a unique $\phi:\ZL\to\,\downarrow\! D(f)$,
  such that ${\phi\circ D}={d(\nicefrac{\_}{1})}$.
  But both $\_\wedge D(f)$ and $\psi_f\circ(\nicefrac{\_}{1})^{\mathcal{L}}$
  satisfy the same commutativity condition as $\phi$, which implies
  that they have to be the same lattice morphism.
  For $u\leq D(f)$ we therefore get
  $\big(\psi_f\circ\psi_f^{-1}\big)(u) = u\wedge D(f) =u$.

  For the converse direction, observe that
  $\psi_f^{-1}\circ d=\psi_f^{-1}\circ\psi_f\circ D$ is a support.
  Consequently, there is a unique morphism such that the following commutes:
  \[
  \begin{tikzcd}
    & \locEl{R}{f} \arrow[dl,"D"']\arrow[dr,"\psi_f^{-1}\circ d~=~\psi_f^{-1}\circ\psi_f\circ D"] & \\
    \ZarLat{\locEl{R}{f}} \arrow[rr,dashed, "\exists!"'] && \ZarLat{\locEl{R}{f}}
  \end{tikzcd}
  \]
  Obviously, $\psi_f^{-1}\circ\psi_f$ satisfies this, but also the identity morphism as
  \begin{align*}
    \big(\psi_f\circ\psi_f^{-1}\big)\big(D(\nicefrac{r}{f^n})\big) = (\nicefrac{\_}{1})^{\mathcal{L}}\big(D(rf)\big) = D(\nicefrac{rf}{1}) = D(\nicefrac{r}{f^n})
  \end{align*}
  since $\nicefrac{f}{1}$ is a unit and $D$ is a support.
  This finishes the proof that $\psi_f$ and $\psi_f^{-1}$ are inverse.
\end{proof}

\section{Spectral schemes as locally ringed lattices}\label{sec:LRDL}
In the previous section we claimed that the Zariski lattice should
be used as the point-free, constructive spectrum of a ring.
Much like in the classical case, a ring is not determined by its
Zariski lattice. Indeed, for any ring $R$ we have
$\ZL\cong\ZarLat{\nicefrac{R}{\mathfrak{N}}}$,
where $\mathfrak{N}=\sqrt{\langle 0\rangle}$
is the nilradical of $R$. In order to recover the original ring, one
needs to equip the Zariski spectrum or lattice with its
\emph{structure sheaf}.  This construction was used as the point of
departure by Coquand, Lombardi and Schuster to develop
``Spectral Schemes as Ringed Lattices'' \cite{ConstrSchemes}.
If one wants to recover ring morphisms from their action on Zariski
lattice and structure sheaf, one needs to introduce some additional
structure making the constructive spectrum a locally ringed
lattice. In this section we want to redevelop key results of
\cite{ConstrSchemes} for ``spectral schemes'' as \emph{locally} ringed lattices.
This gives a rather neat and point-free account of these spectral schemes.
From a classical perspective these are really just qcqs-schemes,
as we will see at the end of this section.
In particular, this lets us relate the lattice theoretic approach to schemes
with the functor of points approach, as we will show in this paper.

\subsection{From ringed lattices to locally ringed lattices}\label{subsec:RDLtoLRDL}%%%%%%%%%%%%%

Let us begin by repeating the definition of a ringed lattice
as given by Coquand, Lombardi and Schuster \cite{ConstrSchemes}.

\begin{definition}\label{def:RDL}
  Let $L$ be a distributive lattice and $\mathcal{F}:L^{op}\to\mathsf{CommRing}$
  a ring-valued presheaf on $L$ (seen as a poset category).
  We say that $\mathcal{F}$ is a sheaf if
  for all $x,u_1,\dots,u_n: L$ with $x=\textstyle\bigvee_{i=1}^n u_i$
  we get a canonical equalizer diagram
  \begin{align*}
    \mathcal{F}(x) \to \prod_{i=1}^n\mathcal{F}(u_i)\rightrightarrows\prod_{i,j}\mathcal{F}(u_i \wedge u_j)
  \end{align*}
  A \emph{ringed lattice lattice} is a distributive lattice with a sheaf of rings.
  Given ringed lattices $(L,\mathcal{F})$ and $(M,\mathcal{G})$,
  a morphism of ringed lattices $\pi:(L,\mathcal{F})\to (M,\mathcal{G})$
  is a pair $\pi:\equiv(\pi^*,\pi^\sharp)$ consisting of a lattice homomorphism
  $\pi^*:L\to M$ and a natural transformation
  $\pi^\sharp:\NatTrans{\mathcal{F}}{\pi_*\mathcal{G}}$,
  where $\pi_*\mathcal{G}:L^{op} \to \mathsf{CommRing}$ is the
  sheaf defined by $\pi_*\mathcal{G}(u):\equiv \mathcal{G}(\pi^*(u))$.
  The category of ringed lattices and their morphisms is denoted $\mathsf{RDL}$.
\end{definition}

The motivation for this definition is based on the classical fact that distributive lattices
correspond to the compact open subsets of so-called coherent or spectral spaces.
An in-depth discussion is included in \cref{subsec:classicalChar}.
Building on the definition of ringed lattice, and with the correspondence
to ringed spectral spaces in mind, the question arises whether
one can define locally ringed lattices corresponding to locally ringed spectral
spaces without appealing to point-set topology or the notion of stalk of a sheaf.

In the context of formal topology,
it was observed by Schuster \cite{SchusterZariski} that
sheaves of local rings can be characterized in a point-free way
by using supports and a second notion, which we will call
invertibility suprema or invertibility opens.
The notion of invertibility opens was used by Hakim
\cite[Sec.\ III.2]{Hakim1972} to define locally ringed \emph{topoi}.
We want to apply these notions to ringed lattices,
following a suggestion by Thierry Coquand.
Coquand observed that one can always define an ``invertibility support''
on a spectral scheme and use it to characterize local morphisms.
We will discuss why this works classically in \cref{subsec:classicalChar}.

\begin{definition}\label{def:InvSup}
  Let $P$ be a poset and $\mathcal{F}:P^{op}\to \mathsf{CommRing}$ a ring-valued presheaf.
  For an object $u: P$ and a section $s:\mathcal{F}(u)$, an element $u_s\leq u$ is
  called an \emph{invertibility supremum} of $s$, if for all $w\leq u$ we have that
  \begin{align*}
    w\leq u_s \quad\Leftrightarrow\quad \rest{s}{w}~\in~\mathcal{F}(w)^\times
  \end{align*}
  In other words, $u_s$ is the largest element smaller than $u$ where the
  restriction of $s$ becomes invertible.

  A dependent function  $\mathcal{D}:(u:P)\to \mathcal{F}(u)\to ~\downarrow\! u$,
  is called an \emph{invertibility map} if $\mathcal{D}_u(s)$ is an invertibility
  supremum of $s$ for all $u: P$ and $s:\mathcal{F}(u)$.
\end{definition}

\begin{remark}\label{rem:isPropInvSup}
  If $u_s$ is an invertibility supremum of $s:\mathcal{F}(u)$, for  $u: P$,
  then clearly $\rest{s}{u_s}\in\mathcal{F}(u_s)^\times$.
  In particular, by anti-symmetry, invertibility suprema are uniquely determined if
  they exist and we will henceforth speak of \emph{the} invertibility supremum of $s$.
  In univalent terms, the type of invertibility suprema of $u$ and $s$ is a proposition.
  It follows that for a poset $P$ and presheaf $\mathcal{F}$, the type of invertibility
  maps is a proposition as well.
\end{remark}

\begin{definition}\label{def:LRDL}
  A \emph{locally ringed lattice} is a distributive lattice $L$ with a sheaf of rings
  $\mathcal{F}$ and an invertibility map
  $\mathcal{D}:(u:L)\to \mathcal{F}(u)\to ~\downarrow\! u$,
  such that $\mathcal{D}_u$ is a support for all $u : L$.
  We call $\mathcal{D}$ the \emph{invertibility support} on $(L,\mathcal{F})$.

  A morphism of locally ringed lattices
  is a morphism of ringed lattices
  $\pi:\equiv(\pi^*,\pi^\sharp):(L,\mathcal{F})\to(M,\mathcal{G})$
  such that for all $u: L$ and $s:\mathcal{F}(u)$:
  \begin{align*}
    \pi^*\big(\Support{u}{s}\big)\;=\;\Support{\pi^*(u)}{\pi^\sharp(s)}
  \end{align*}
  The category of locally ringed distributive lattices will be denoted by \textsf{LRDL}.
\end{definition}

\begin{remark}\label{rem:dropSupportLRDL}
  Although invertibility supports are given as additional structure on ringed lattices,
  we will write locally ringed lattices as pairs $(L,\mathcal{F})$ and always denote the
  implicit invertibility support by $\mathcal{D}$. It will always be clear
  from context which locally ringed lattice an invertibility support corresponds to.
  This notation is motivated by the fact that invertibility supports
  are uniquely determined (if they exist), as explained in \cref{rem:isPropInvSup}.
\end{remark}

\begin{remark}
  If $\pi:(L,\mathcal{F})\to(M,\mathcal{G})$ is a morphism of
  ringed lattices, then for $u: L$ and $s:\mathcal{F}(u)$ one always
  has $\pi^*\big(\Support{u}{s}\big)\leq\Support{\pi^*(u)}{\pi^\sharp(s)}$,
  as by the naturality of  $\pi^\sharp$:
  \begin{align*}
    \rest{\pi^\sharp(s)}{\pi^*(\Support{u}{s})} =\pi^\sharp\big(\rest{s}{\Support{u}{s}}\big)~\in~\mathcal{G}\big(\pi^*(\Support{u}{s})\big)^\times
  \end{align*}
  $\pi$ is thus a morphism of locally ringed lattices if and only if the converse
  inequality holds for all $u$ and $s$.
\end{remark}

Nothing in the definition of locally ringed lattices depends on univalent features
and the notion can and should be studied
in other settings than HoTT/UF, such as Bishop-style constructive mathematics.
What is, however, important for our purposes, is that it is a structurally
well-behaved notion in the sense that for locally ringed lattices
``isomorphisms are equalities''. This will actually be a key fact in the proof
of our main comparison theorem.

\begin{proposition}\label{prop:isUnivalentLRDL}
  $\mathsf{LRDL}$ is a univalent category.
\end{proposition}

\begin{proof}
  Instead of proving that locally ringed lattices and their isomorphisms
  are a ``standard notion of structure'' \cite[Def.\ 9.8.1]{HoTTBook} directly,
  it is easier to use a ``displayed'' approach \cite{DisplayedCats,DURGs}.\footnote{
    The \texttt{agda/cubical}-library uses the approach of \cite{DURGs} and supports
    some automation for proving that structures are univalent.
    An earlier version described in \cite{POPLPaper} was based on
    \cite{Escardo22} that also has a modular approach to building univalent structures,
    albeit not a displayed one. Formalizing locally ringed lattices and proving them univalent
    could figure as an interesting test case for the proof automation tools,
    as it involves two layers of displayedness.}
  It is not hard to show that bounded distributive lattices with
  lattice morphisms form a univalent category and lattice sheaves with
  natural transformations form a univalent displayed category on
  lattices. Finally, invertibility supports, with their canonical
  notion of (displayed) morphism being exactly that of a morphism of
  locally ringed lattices, form a univalent displayed category on
  ringed lattices.
\end{proof}

\begin{corollary}
  The forgetful functor $\mathcal{U}$ from locally ringed lattices to ringed lattices
  reflects isomorphisms.
\end{corollary}

\begin{proof}
  Since the type of invertibility supports over a ringed lattice is a proposition,
  the structure identity principle implies that both of the canonical maps
  $X=Y\to X\cong_{\mathsf{LRDL}}Y$ and
  $X=Y\to \mathcal{U}(X)\cong_{\mathsf{RDL}}\mathcal{U}(Y)$
  are equivalences for $X,Y:\mathsf{LRDL}$.
\end{proof}

\begin{example}\label{ex:affSchLRDL}
  The main example of a locally ringed lattice will of course
  be the ``constructive spectrum'' of a ring.
  For $R:\mathsf{CommRing}$ we can equip the Zariski lattice $\ZL$
  with its structure sheaf $\mathcal{O}_R$. The structure
  sheaf is defined on basic opens and then lifted to the entire Zariski lattice.
  The basic opens as a subset
  $\mathcal{B}_R:\equiv\{u:\ZL~\vert~ \exists_{f:R}~ u=D(f)\}\subseteq\ZL$
  form a basis of the Zariski lattice in the sense that for every $u:\ZL$
  there merely exist $b_1,...,b_n:\mathcal{B}_R$ such that $u=\bigvee_{i=1}^n b_i$.\footnote{
    As $\ZL$ is defined as a quotient, we know that merely $u=D(f_1,...,f_n)=\bigvee_{i=1}^n D(f_i)$.}
  By the \emph{comparison lemma for distributive lattices} \cite[Lemma 1]{ConstrSchemes}\footnote{
    This is a special case of the general comparison lemma for sites
    \cite[Cor.\ 3, p.\ 590]{SheavesInGeometryAndLogic}. The lattice case is formalized
    in \texttt{Cubical Agda} and can be found in:
    \url{https://github.com/agda/cubical/blob/master/Cubical/Categories/DistLatticeSheaf/ComparisonLemma.agda}.}
  there is an equivalence of categories between sheaves on $\ZL$ and sheaves on the basis
  $\mathcal{B}_R$. Here, sheaves on a basis of a distributive lattice
  are defined completely analogously to \cref{def:RDL} and morphisms
  of sheaves are just natural transformations.
  Given a sheaf $\mathcal{F}$ on $\mathcal{B}_R$, the equivalence of
  the comparison lemma maps this to a sheaf on $\ZL$ by taking the
  \emph{right Kan extension} $\mathsf{Ran}(\mathcal{F})$ along the
  basis inclusion.

  As was shown in the \texttt{Cubical Agda} formalization of the structure sheaf
  described in \cite{ZeunerMortberg23},
  we can use univalence to define the structure sheaf on
  basic opens by mapping $D(f)\mapsto\locEl{R}{f}$.\footnote{
    In \cite[Sec.\ 2.3]{ConstrSchemes}, the structure
    sheaf on basic opens is defined using localization at the saturation
    $S_f=\{g\,\vert\, D(f)\leq D(g)\,\}$.}
  In this paper we do not want to get into the details of how the
  structure sheaf is defined and just work with a sheaf
  $\mathcal{O}_R:(\ZL)^{op}\to\mathsf{CommRing}$, such that
  $\mathcal{O}_R(D(f))\cong\locEl{R}{f}$ canonically.  This means that
  these isomorphisms satisfy all of the expected functoriality and
  naturality conditions. Note however, that definitions and proofs
  appealing to the canonical isomorphism need to be rephrased to
  accommodate for the formal definition of the structure sheaf, before
  becoming formalizable themselves.

  Observe that in any locally ringed lattice one has
  $w\wedge\Support{u}{s}~=~\Support{w}{\rest{s}{w}}$, for $w\leq u$
  and thus for joins we must have
  \begin{align*}
    \Support{u_1\vee\dots\vee u_n}{s}~=~\textstyle\bigvee_{i=1}^n\Support{u_i}{\rest{s}{u_i}}
  \end{align*}
  In order to define an invertibility support on $(\ZL,\mathcal{O}_R)$,
  it is thus sufficient to define $\mathcal{D}$ on the basic opens.
  Modulo the canonical isomorphism we can set:
  \begin{align*}
    \Support{D(f)}{\nicefrac{r}{f^n}}~:\equiv~D(fr) \quad(=D(f)\wedge D(r))
  \end{align*}
  Indeed, this gives us the invertibility supremum, since
  $\locEl{R}{fr}\cong{\locEl{\locEl{R}{f}}{r}}$
  is the initial $\locEl{R}{f}$-algebra where $\nicefrac{r}{1}$ becomes invertible.\footnote{Note that $\mathcal{D}_{D(f)}$ is the support $d$ in the proof of \cref{lem:isoLocDownSet}.}
  For arbitrary elements of the Zariski lattice we can then set:
  \begin{align*}
    \Support{D(f_1,\dots,f_n)}{s}~:\equiv~\textstyle\bigvee_{i=1}^n\Support{D(f_i)}{\rest{s}{D(f_i)}}
  \end{align*}
  By the above remark, this gives us an invertibility map.\footnote{The well-definedness of this
    construction follows from the fact that
    the type of invertibility maps is a proposition.}
  To check the support conditions we can again restrict our attention
  to basic opens where everything follows the fact that $D$ is a support.
  Note that $\mathcal{D}_{D(1)}=D$ modulo the canonical $\mathcal{O}_R(D(1))\cong R$.
\end{example}

\subsection{Universal property} %%%%%%%%%%%%%%%%%%%%%%%%%%%%%%%%%%%%%%%%%%%%%%%%%%%%%%%%%%%%%%%%%%%%%%

In classical algebraic geometry, the following proposition is sometimes called
the ``universal property of schemes'':
\begin{quote}
  \emph{The functor \textsf{Spec} is left-adjoint to the global sections functor $\Gamma$
  and the unit of this adjunction is an isomorphism.}
\end{quote}
This is of course already true when regarding \textsf{Spec}
as a functor from rings to locally ringed spaces \cite[Prop 1.6.3]{EGA1},
Schuster \cite{SchusterZariski} gives a constructive proof
of a point-free version of this statement using formal topology.\footnote{
  Curiously, previous proofs in the literature, such as the one given by Johnstone in
  ``Stone Spaces'' \cite[V.3.5]{StoneSpaces}, appeal to classical reasoning by
  using the points of \textsf{Spec}, even if the Zariski spectrum is defined in a point-free way.}
We want to show the corresponding statement for locally ringed lattices.

Let us start by addressing the issue of variance that arises
when the spectrum functor \textsf{Spec} is introduced
as a contra-variant functor from rings to locally ringed spaces.
This makes the direction of the adjunction somewhat arbitrary.
In the functorial setting of \cref{prop:OSpRelCoadj}
it is more natural to think of the spectrum functor as the right adjoint.
In order to match this we will thus work in opposite categories
$\mathsf{CommRing}^{op}$ and $\mathsf{LRDL}^{op}$.

This actually makes a lot of sense for morphisms locally ringed lattices.
As is explained in \cref{subsec:classicalChar}, readers familiar
with classical algebraic geometry should think of a locally ringed lattice
$(L_X,\mathcal{O}_X)$ as a so-called spectral space $X$ with a sheaf of rings
that is local on the stalks. The lattice $L_X$ is the lattice of
compact open subsets of $X$. A spectral morphism of suitable locally ringed
spectral spaces $f:X\to Y$ induces a lattice morphism $f^*:L_Y\to L_X$,
by taking the pre-image of compact opens of $Y$ (spectral morphism means
that pre-images of compact opens are compact open).

For an object in $X:\mathsf{LRDL}^{op}$
we denote its lattice by $L_X$ and its sheaf by $\mathcal{O}_X$.
A morphism $\pi:\catHom{\mathsf{LRDL}^{op}}{X}{Y}$ is given
by the lattice morphism $\pi^*:L_Y\to L_X$,
similar to how a morphism of locales is given by a frame homomorphism in
the opposite direction, and the natural transformation
$\pi^\sharp:\NatTrans{\mathcal{O}_Y}{\pi_*\mathcal{O}_X}$.

\begin{definition}\label{def:globalSecFun}
  The \emph{global sections functor}
  $\Gamma:\mathsf{LRDL}^{op}\to\mathsf{CommRing}^{op}$ is defined on
  objects by $\Gamma (L_X,\mathcal{O}_{X}):\equiv\mathcal{O}_{X}(1)$
  and on morphisms $\pi:\catHom{\mathsf{LRDL}^{op}}{X}{Y}$ by
  $\Gamma(\pi):\equiv\pi^\sharp_1$ modulo the identification
  $\mathcal{O}_X(\pi^*(1))=\mathcal{O}_X(1)$.
\end{definition}

\begin{definition}\label{def:SpecFun}
  The functor $\mathsf{Spec}:\mathsf{CommRing}^{op}\to\mathsf{LRDL}^{op}$ is defined
  on objects by $\mathsf{Spec}(R):\equiv(\mathcal{L}_R,\mathcal{O}_R)$
  as described in \cref{ex:affSchLRDL}.
  For a ring morphism $\varphi:\mathsf{Hom}(R,A)$, we set
  $\mathsf{Spec}(\varphi)^*:\equiv\varphi^{\mathcal{L}}:\ZarLat{R}\to\ZarLat{A}$.
  To define the natural transformation $\mathsf{Spec}(\varphi)^\sharp$, we only
  have to define it on basic opens. For $f:R$, we set modulo the canonical
  isomorphisms $\mathcal{O}_R(D(f))\cong\locEl{R}{f}$ and
  $\mathcal{O}_A(\varphi^{\mathcal{L}}(D(f)))=\mathcal{O}_A(D(\varphi(f)))\cong\locEl{A}{\varphi(f)}$
  \[
  \begin{tikzcd}
    & R \arrow[dl,"\nicefrac{\_}{1}"']\arrow[dr,"\nicefrac{\varphi(\_)}{1}"] & \\
     \locEl{R}{f}\arrow[rr,dashed, "\exists!~\mathsf{Spec}(\varphi)^\sharp"'] && \locEl{A}{\varphi(f)}
  \end{tikzcd}
  \]
  One easily checks that this defines a morphism of locally ringed lattices.
\end{definition}

\begin{theorem}\label{thm:GammaSpecAdj}
  \textsf{Spec} is right adjoint to $\Gamma$ and the counit of this adjunction is an isomorphism.
\end{theorem}

\begin{proof}
  Given a commutative ring $R$ and a locally ringed lattice $X:\equiv(L_X,\mathcal{O}_{X})$, there
  is a canonical function
  \begin{align*}
    \catHom{\mathsf{CommRing}^{op}}{\mathcal{O}_X(1)}{R}\to\catHom{\mathsf{LRDL}^{op}}{X}{\mathsf{Spec}(R)}
  \end{align*}
  Let $\varphi:\mathsf{Hom}(R,\mathcal{O}_L(1))$ be a ring homomorphism. Note that
  $\mathcal{D}_1\circ\,\varphi:R\to L_X$ is a support. This induces a lattice homomorphism
  \[
  \begin{tikzcd}
    & R \arrow[dl,"D"']\arrow[dr,"\mathcal{D}_1\circ\,\varphi"] & \\
    \mathcal{L}_R \arrow[rr,dashed, "\exists!~\pi^*"'] && L
  \end{tikzcd}
  \]
  Furthermore, for every $f:R$ we get a map
  \[
  \begin{tikzcd}
    & R \arrow[dl]\arrow[dr,"\rest{\varphi(\_)\,\,}{\Support{1}{\varphi(f)}}"] & \\
    \locEl{R}{f} \arrow[rr,dashed, "\exists!~\pi^\sharp"'] && \mathcal{O}_L\big(\underbrace{\Support{1}{\varphi(f)}}_{=~\pi^*(D(f))}\big)
  \end{tikzcd}
  \]
  This defines a natural transformation between sheaves on basic opens
  and thus on $\mathcal{L}_R$ by the comparison lemma for distributive
  lattices.  It remains to check that this defines a morphism of
  locally ringed lattices.  Note that on global sections we have by
  the definition of $\pi^*$ and $\pi^\sharp$ that
  \begin{align*}
    \pi^*\big(D(f)\big) ~=~ \Support{1}{\varphi(f)} ~=~ \Support{\pi^*(D(1))}{\pi^\sharp(f)}
  \end{align*}
  modulo identifying $\mathcal{O}_R(D(1))$ with $R$. Thus on arbitrary basic opens,
  we have
  \begin{align*}
    \pi^*\big(\Support{D(f)}{\nicefrac{r}{f^n}}\big)
    ~&=~ \Support{1}{\pi^\sharp(rf)} \\
    ~&=~ \Support{1}{\pi^\sharp(r)}\wedge\Support{1}{\pi^\sharp(f)}
    ~=~ \Support{\Support{1}{\pi^\sharp(f)}}{\pi^\sharp(\nicefrac{r}{f^n})}
  \end{align*}
  and finally on arbitrary elements
  \begin{align*}
    \pi\big(\Support{D(f_1)\vee\dots\vee D(f_n)}{s}\big)
    ~&=~\textstyle\bigvee_{i=1}^n\Support{\pi(D(f_i))}{\pi^\sharp(s)\restriction_{\pi(D(f_i))}} \\
    ~&=~\Support{\pi(D(f_1)\vee\dots\vee D(f_n))}{\pi^\sharp(s)}
  \end{align*}
  Note that the action of $\mathsf{Spec}$ on morphisms is essentially the same construction.
  We omit the computations for showing that this is natural in both $R$ and $X$.

  The inverse of this map is given by $\Gamma$. Indeed,
  it follows directly that for a ring morphism $\varphi:\mathsf{Hom}(R,\mathcal{O}_L(1))$
  one has $\varphi=\pi^\sharp$ on global sections.
  For the other direction fix a morphism
  $\pi:\catHom{\mathsf{LRDL}^{op}}{X}{\mathsf{Spec}(R)}$ and observe
  that the following diagram commutes for $\pi^\sharp$ at global sections.
  \[
  \begin{tikzcd}
    & R \arrow[dl,"D"']\arrow[dr,"\mathcal{D}_1\circ\,\pi^\sharp"] & \\
    \mathcal{L}_R \arrow[rr, "\pi^*"'] && L_X
  \end{tikzcd}
  \]
  as well as for $f:R$
  \[
  \begin{tikzcd}
    & R \arrow[dl]\arrow[dr,"\pi^\sharp(\_)\restriction_{\Support{1}{\pi^\sharp(f)}}"] & \\
    \locEl{R}{f} \arrow[rr,"\pi^\sharp"'] && \mathcal{O}_L\big(\pi(D(f))\big)
  \end{tikzcd}
  \]
  This shows that $\pi$ is just the morphism
  obtained from applying the adjunction map to $\pi^\sharp$ at global sections.
  From this we also get that the counit is an isomorphism.
  Indeed, for a ring $R$ it is the canonical isomorphism ${R}\cong\mathcal{O}_R(D(1))$.
\end{proof}

\begin{corollary}\label{cor:isFullyFaithfulSpec}
  \textsf{Spec} is fully faithful.
\end{corollary}

\subsection{Spectral schemes}\label{subsec:SpectralSch} %%%%%%%%%%%%%%%%%%%%%%%%%%%%%%%%%%%%%%%%%%%%%

Let $X$ be a locally ringed lattice and $R:\equiv\mathcal{O}_X(1)$
be the ring of global sections. We saw that the counit of the adjunction
of \cref{thm:GammaSpecAdj} is an isomorphism. But what about its unit?
Following the proof of \cref{thm:GammaSpecAdj} for  the identity on $R$,
we see that the lattice homomorphism is induced by
\[
\begin{tikzcd}
  & R \arrow[dl,"D"']\arrow[dr,"\mathcal{D}_1"] & \\
  \mathcal{L}_R \arrow[rr,dashed, "\exists!~\eta^*_X"'] && L_X
\end{tikzcd}
\]
and the natural transformation $\eta^\sharp_X:\mathcal{O}_R\to (\eta_X)_*\mathcal{O}_X$
is given on a basic open $D(f)$, $f:R$, by
\[
\begin{tikzcd}
  & R \arrow[dl]\arrow[dr,"\_\restriction_{\Support{1}{f}}"] & \\
  \locEl{R}{f} \arrow[rr,dashed, "\exists!~\eta^\sharp_X"'] && \mathcal{O}_X\big(\underbrace{\Support{1}{f}}_{=~\eta^*_X(D(f))}\big)
\end{tikzcd}
\]

\begin{definition}\label{def:affSchUnit}
  $X$ is an  \emph{affine scheme}, if $\eta_X$ is an isomorphism
  of locally ringed lattices.
\end{definition}

\begin{lemma}\label{lem:isAfffineSpec}
  $\mathsf{Spec}(R)$ is affine.
\end{lemma}

\begin{lemma}\label{lem:affSchBig}
  $X$ is affine if and only if there exists a ring $A$ such that
  $X\cong\mathsf{Spec}(A)$ as locally ringed lattices.
\end{lemma}

\begin{proof}
  \textsf{Spec} is fully faithful.
\end{proof}

\begin{remark}
  The advantage of \cref{def:affSchUnit}, compared to characterizing
  affine schemes by the equivalent statement of \cref{lem:affSchBig},
  is that the latter is a ``big'' proposition living in the next higher
  universe as it quantifies over the type of all rings.
\end{remark}

For a locally ringed lattice $X$ and $u:L_X$, we get a locally ringed lattice
$\rest{X}{u}:\equiv(\downarrow\! u,~\rest{\mathcal{O}_X}{\downarrow u})$
of elements below $u$ with the sheaf $\mathcal{O}_X$ restricted to those elements.
The lattice morphism $\_\wedge u:L_X\to\,\downarrow\! u$ induces a
map $\catHom{\mathsf{LRDL}^{op}}{\rest{X}{u}}{X}$, which we think of as the ``inclusion''
of the compact open $u$ into $X$.

\begin{definition}\label{def:qcqsSch}
  $X$ is a  \emph{qcqs-scheme}, if there merely exist $u_1,\dots,u_n: L_X$,
  which cover $X$ in the sense that $1=u_1\vee\dots\vee u_n$ and such that
  $\rest{X}{u_i}$ is affine
  for all $i$ in $1,...,n$.
\end{definition}

\begin{example}
  Let us discuss the standard example of a qcqs-scheme that is not affine:
  The plane without the origin $\mathbb{A}^2\setminus\{0\}$ over some fixed
  field $k$.
  Classically, the points of $\mathbb{A}^2\setminus\{0\}$ correspond to
  $D(x,y)\subseteq\mathsf{Spec}\big(k[x,y]\big)$.
  The goal is thus to prove that
  $\rest{\mathsf{Spec}\big(k[x,y]\big)}{D(x,y)}$ is not affine, where $D(x,y):\ZarLat{k[x,y]}$.
  By the sheaf property of the structure sheaf, we get the following pullback square
  \begin{center}
    \begin{tikzcd}
      \mathcal{O}_{k[x,y]}(D(x,y))\ar[r]\ar[d]\pbsign{dr} &k[x,y,y^{-1}]\ar[d] \\
      k[x,y,x^{-1}]\ar[r] &k[x,y,x^{-1},y^{-1}]
    \end{tikzcd}
  \end{center}
  Since this is a pullback of integral domains
  we can identify $\mathcal{O}_{k[x,y]}\big(D(x,y)\big)$
  with the subring $k[x,y,x^{-1}]\cap k[x,y,y^{-1}]\subseteq k[x,y,x^{-1},y^{-1}]$.
  One can show that this already gives us a unique
  isomorphism of $k$-algebras $k[x,y]\cong\mathcal{O}_{k[x,y]}\big(D(x,y)\big)$.
  If $\rest{\mathsf{Spec}\big(k[x,y]\big)}{D(x,y)}$
  was affine we would get an isomorphism of lattices
  \[
  \begin{tikzcd}
    & k[x,y] \arrow[dl,"D"']\arrow[dr,"(D(x)\vee D(y))~\wedge~ D(\_)"] & \\
    \mathcal{L}_{k[x,y]} \arrow[rr,dashed, "\exists!"',"\simeq"] && \downarrow D(x,y)
  \end{tikzcd}
  \]
  This morphism has to be $D(x,y)\wedge\_$, which is not an isomorphism, giving the desired contradiction.
\end{example}

It was already shown in \cite{ConstrSchemes} that qcqs-schemes are
obtained by ``gluing'' together affine schemes. For a cover
$u_1,...,u_n$ of $X$ we can glue along the corresponding
``inclusions'' $\rest{X}{u_i}\to X$ to recover $X$.  This was shown
for ringed lattices in \cite[Lemma 2]{ConstrSchemes}, but the proof
extends directly to locally ringed lattices.  We thus give the
following proposition without proof.

\begin{proposition}\label{prop:gluingLatticeQcqsSchemes}
  If $X$ is a qcqs-scheme and $u_1,...,u_n:L_X$
  an (affine) cover of $X$, then
  \begin{align*}
    X~\cong~\mathsf{colim}\,\big\{\rest{X}{u_i}\leftarrow \rest{X}{u_i\wedge u_j}\rightarrow \rest{X}{u_j}\big\}
  \end{align*}
  in the opposite category of locally ringed lattices $\mathsf{LRDL}^{op}$.
  In particular, if we have $\rest{X}{u_i}\cong\mathsf{Spec}(A_i)$,
  and we are given affine covers
  $\big(v_{ijk}:~\downarrow\!(u_i\wedge u_j)\big)_k$
  of the subschemes $\rest{X}{u_i\wedge u_j}$
  with $\rest{X}{v_{ijk}}\cong\mathsf{Spec}(A_{ijk})$,
  we have
  \begin{align*}
    X~\cong~\mathsf{colim}\,\big\{\mathsf{Spec}(A_i)\leftarrow \mathsf{Spec}(A_{ijk})\rightarrow \mathsf{Spec}(A_j)\big\}
  \end{align*}
\end{proposition}

\begin{corollary}
  Let $X$ be a qcqs-scheme,
  then $\mathcal{O}_X\big(\Support{u}{s}\big)\cong\locEl{\mathcal{O}_X(u)}{s}$,
  for any $u: L_X$ and $s:\mathcal{O}_X(u)$.\footnote{This is also known as the ``Qcqs-Lemma'' in ``The Rising Sea'' \cite{RisingSea}.}
\end{corollary}

We conclude this section on our first definition of qcqs-schemes
by showing that the notion of morphism of spectral schemes
given in \cite[Def.\ 16]{ConstrSchemes} coincides with morphisms of locally ringed lattices.

\begin{definition}\label{def:locAffineMor}
  Let $X:\equiv(L_X,\mathcal{O}_X)$ and $Y:\equiv(L_Y,\mathcal{O}_Y)$ be two qcqs-schemes,
  a morphism of \emph{ringed} lattices $\pi:\catHom{\mathsf{RDL}^{op}}{X}{Y}$
  is called \emph{locally affine} if there are affine covers $u_1,\dots,u_n$ of $L_X$ and
  $w_1,\dots,w_m$ of $L_Y$ respectively, which are compatible in the following way:
  For every $u_i$ exists a $w_j$ with $u_i\leq\pi^*(w_j)$ such that the
  morphism of ringed lattices
  $\rest{\pi}{u_i}:\catHom{\mathsf{RDL}^{op}}{\rest{X}{u_i}}{\rest{Y}{w_j}\!\!}$,
  which is induced by the lattice morphism
  $\pi^*(\_)\wedge u_i :\,\downarrow\! w_j\to\, \downarrow\! u_i$,
  is just $\mathsf{Spec}$ applied to the ring morphism
  $\rest{\pi^\sharp}{u_i}:\mathcal{O}_Y(w_j)\to\mathcal{O}_X(u_i)$,
  i.e.\ the following diagram commutes:
  \[\begin{tikzcd}
          {\rest{X}{u_i}} && {\rest{Y}{w_j}} \\
          \\
          {\mathsf{Spec}\big(\mathcal{O}_X(u_i)\big)} && {\mathsf{Spec}\big(\mathcal{O}_Y(w_j)\big)}
          \arrow["{\rest{\pi\,}{u_i}}", from=1-1, to=1-3]
          \arrow["{\eta_{\rest{X}{u_i}}}"', "\vvcong", from=1-1, to=3-1]
          \arrow["{\eta_{\rest{Y}{w_j}}}","\vcong"', from=1-3, to=3-3]
          \arrow["{\mathsf{Spec}(\rest{\pi^\sharp\,}{u_i})}", from=3-1, to=3-3]
  \end{tikzcd}\]
\end{definition}

\begin{remark}\label{rem:locAffineMorRemark}
  Since $\mathsf{Spec}$ is fully-faithful, a morphism
  $\pi:\catHom{\mathsf{RDL}^{op}}{X}{Y}$
  is locally affine if and only if there are compatible affine covers
  $u_1,\dots,u_n$ of $L_X$ and $w_1,\dots,w_m$ of $L_Y$
  such that for  $u_i\leq\pi^*(w_j)$ the induced morphism
  $\rest{\pi}{u_i}:\catHom{\mathsf{RDL}^{op}}{\rest{X}{u_i}}{\rest{Y}{w_j}\!\!}$
  is a morphism of locally ringed lattices.
\end{remark}

\begin{lemma}\label{lem:locAffMorEquiv}
  Let $X$ and $Y$ be two qcqs-schemes,
  then any morphism of ringed lattices $\pi:\catHom{\mathsf{RDL}^{op}}{X}{Y}$
  is locally affine if and only if it is a morphism of locally ringed lattices.
\end{lemma}

\begin{proof}
  First, assume that $(\pi^*,\pi^\sharp)$ is a morphism of locally ringed lattices.
  take affine covers $u_1,\dots,u_n$ of $L_X$ and $w_1,\dots,w_m$ of $L_Y$.
  We can refine the $u_i$'s to a compatible cover in the following way.
  For each $i$ and $j$ take a cover $\pi^*(w_j)\wedge u_i=\bigvee_k\Support{u_i}{s_{ijk}}$.
  Such a cover exists because $u_i$ is affine and all $\Support{u_i}{s_{ijk}}$ are
  affine as well, i.e.\
  $\rest{X}{\Support{u_i}{s_{ijk}}}=\mathsf{Spec}(\locEl{\mathcal{O}_X(u_i)}{s_{ijk}})$.
  Then $1_{L_X}=\bigvee_{i,j,k}\Support{u_i}{s_{ijk}}$, with $\Support{u_i}{s_{ijk}}\leq\pi^*(w_j)$
  and the induced/restricted morphism $\rest{X}{\Support{u_i}{s_{ijk}}}\to \rest{Y}{w_j}$
  is one of locally ringed lattices, hence affine.

  Now, assume that we have compatible covers $u_1,\dots,u_n$ of $L_X$ and $w_1,\dots,w_m$ of $L_Y$,
  with $u_i\leq\pi^*(w_j)$ such that
  $\rest{\pi}{u_i}:\catHom{\mathsf{RDL}^{op}}{\rest{X}{u_i}}{\rest{Y}{w_j}\!\!}$
  is actually a morphism of locally ringed lattices.
  This means that for $v\leq w_j$ and $s:\mathcal{O}_Y(v)$ we have
  \begin{align*}
    u_i\wedge\pi^*\big(\Support{v}{s}\big)=\Support{u_i\wedge\pi^*(v)}{\rest{\pi^\sharp(s)}{u_i\wedge\pi^*(v)}}=u_i\wedge\Support{\pi^*(v)}{\pi^\sharp(s)}
  \end{align*}
  For $w: L_Y$ and $s:\mathcal{O}_Y(w)$ we have that
  $u_i\wedge\pi^*(w)=u_i\wedge\pi^*(w_j\wedge w)$, since $u_i\leq\pi^*(w_j)$,
  and hence
  \begin{align*}
        \Support{u_i\wedge\pi^*(w)}{\rest{\pi^\sharp(s)}{u_i\wedge\pi^*(w)}}
    ~&=~ u_i\wedge\Support{\pi^*(w_j\wedge w)}{\rest{\pi^\sharp(s)}{\pi^*(w_j\wedge w)}} \\
    ~&=~ u_i\wedge\pi^*\big(\Support{w_j\wedge w}{\rest{s}{w_j\wedge w}}\big) \\
    ~&\leq~ \pi^*\big(\Support{w_j\wedge w}{\rest{s}{w_j\wedge w}}\big)
  \end{align*}
  and hence, since for each $i$ there exists a compatible $j$,
  \begin{align*}
    \Support{\pi^*(w)}{\pi^\sharp(s)} ~&=~ \textstyle\bigvee_{i=1}^n\Support{u_i\wedge\pi^*(w)}{\rest{\pi^\sharp(s)}{u_i\wedge\pi^*(w)}} \\
    ~&\leq~ \textstyle\bigvee_{j=1}^m\pi^*\big(\Support{w_j\wedge w}{\rest{s}{w_j\wedge w}}\big) \\
    ~&=~ \pi^*\big(\Support{w}{s}\big) \qedhere
  \end{align*}
\end{proof}

\subsection{Classical characterization}\label{subsec:classicalChar}%%%%%%%%%%%%%%%%%%%%%%%%%%%%%%%

Before moving on, we want to sketch why
the definition of qcqs-schemes as locally ringed lattices given above
is actually equivalent to the standard definition using locally ringed spaces.
This section is addressed to readers familiar with classical algebraic geometry.
For this section only, we leave the realm of constructive type theory
and adopt set-theoretic notation.

First, it is worth checking that invertibility supports are related to sheaves of local rings.
Observe that any ringed space $(X,\mathcal{O}_X)$ has arbitrary invertibility suprema.
For any open $U\subseteq X$ and $s\in\mathcal{O}_X(U)$, we can define
\begin{align*}
  \Support{U}{s}=\big\{x\in U~\vert~s_x\in\mathcal{O}_{X,x}^{\,\times}\big\}
\end{align*}
$\mathcal{D}$ is easily seen to be an invertibility map
(on opens of $X$ and $\mathcal{O}_X$).
$\Support{U}{s}$ is sometimes called the \emph{invertibility open} of s.\footnote{
See \cite[Def.\ 2.1]{Hakim1972}.}
Moreover, it contains all information about locality on the stalks,
as the following propositions show.

\begin{proposition}
  $(X,\mathcal{O}_X)$ is a locally ringed space if and only if $\mathcal{D}_U$ is a support
  for all opens $U\subseteq X$.
\end{proposition}

\begin{proof}
  Note that we always have
  \begin{align*}
    \Support{U}{1}=U\quad\&\quad\forall s,t:\;\Support{U}{st}=\Support{U}{s}\cap\Support{U}{t}
  \end{align*}
  So what we actually want to prove is that $(X,\mathcal{O}_X)$ is a locally ringed space
  if and only if for all open sets $U$ we have:
  \begin{align*}
    \Support{U}{0}=\emptyset\quad\&\quad\forall s,t:\;\Support{U}{s+t}\subseteq\Support{U}{s}\cup\Support{U}{t}
  \end{align*}
  Recall that a ring $A$ is local if and only if it is non-trivial and
  the non-units $A\setminus A^\times$ form an ideal, which is the case if and only if
  $A\setminus A^\times$ is closed under addition.
  Taking the contrapositive of the last statement, we arrive at:
  \emph{$A$ is local iff it is non-trivial and $s+t\in A^\times$
  implies $s\in A^\times$ or $t\in A^\times$}.\footnote{In fact, this is usually taken to be the constructive definition of a local ring.}
  Using this characterization for the stalks $\mathcal{O}_{X,x}$ the claim immediately follows.
\end{proof}

\begin{proposition}
  $(f,f^\sharp):(X,\mathcal{O}_X)\to (Y,\mathcal{O}_Y)$ is a morphism of locally ringed spaces
  if and only if for every open $U\subseteq Y$ and section $s\in\mathcal{O}_Y(U)$:\footnote{In
    \cite[Def.\ 2.9]{Hakim1972}, a morphism of ringed topoi satisfying the
    analoguous condition is called \emph{admissible}.}
  \begin{align*}
    f^{-1}\big(\Support{U}{s}\big)\;=\;\Support{f^{-1}(U)}{f^\sharp(s)}
  \end{align*}
\end{proposition}

\begin{proof}
  This follows from the fact that a homomorphism of local rings $\varphi:A\to B$
  is local if and only if $\varphi(a)\in B^\times$ implies $a\in A^\times$.
\end{proof}

The next step is to establish a connection between (ringed) lattices and (ringed) spaces.
The main result that makes this possible is
\emph{Stone's representation theorem for distributive lattices}~\cite{StoneRepresentation}.
It gives a (contravariant) equivalence of categories between distributive lattices
and a special class of topological spaces, called \emph{coherent} or \emph{spectral} spaces.
A topological space $X$ is coherent if it is \emph{quasi-compact}, \emph{sober}
(its non-empty irreducible closed subsets are the closure of a single point),
and its quasi-compact opens are closed under finite intersections
and form a basis of the topology of $X$.
A coherent map between coherent spaces
$X$ and $Y$ is a continuous map $f:X\to Y$ such that for any quasi-compact open $K\subseteq Y$,
its pre-image $f^{-1}(K)$ is quasi-compact.

The category $\mathsf{CohSp}$ consists of coherent spaces with coherent maps.
The equivalence $\mathsf{CohSp}\simeq\mathsf{DL}^{op}$ sends a coherent space $X$
to its lattice of quasi-compact opens $\mathbf{K}^o(X)$ and a coherent map
$f:X\to Y$ to the pre-image lattice homomorphism $f^{-1}:\mathbf{K}^o(Y)\to\mathbf{K}^o(X)$.
The inverse can also be defined explicitly, assigning to a distributive lattice $L$
the space of points of the locale of $L$-\emph{ideals} $\text{pt}\,\big(\mathsf{Idl}(L)\big)$.
See~\cite[Sec.\ II.3.3]{StoneSpaces} for details.

By extension of Stone's representation theorem,
ringed lattices are equivalent to ringed coherent spaces.
Sheaves on a topological space $X$ are in bijection with sheaves on
any basis of $X$ by the comparison lemma for topological spaces.
Hence, for any coherent space $X$ we get an equivalence of categories
\begin{align*}
  \mathsf{Sh}\big(\mathbf{K}^o(X)\big) ~\simeq~ \mathsf{Sh}(X)
\end{align*}
Since the quasi-compact opens are all quasi-compact, we only have to consider
finite covers for the sheaf property and we can thus identify
$\mathsf{Sh}\big(\mathbf{K}^o(X)\big)$ with the category of lattice sheaves
on $\mathbf{K}^o(X)$.

However, in the locally ringed case things get a bit more complicated.
In \cref{def:LRDL} we had to add the invertibility map as an extra
structure on the ringed lattice that could not be defined
automatically. This is for the subtle reason that we required
$\Support{u}{s}$ to be an element of the lattice. In other words,
locally ringed lattices correspond to
locally ringed spaces $(X,\mathcal{O}_X)$
where $X$ is a coherent space and \emph{for any quasi-compact open $U\subseteq X$
and section $s:\mathcal{O}_X(U)$, the invertibility open
$\Support{U}{s}$ is quasi-compact as well}.
This last condition need not always hold,
as the following counterexample due to Thierry Coquand shows:

\begin{example}
  The one-point compactification of the naturals (as a discrete space)
  $X=\mathbb{N}\cup\{\infty\}$ is coherent.\footnote{It is also Hausdorff and hence
    a Stone space, corresponding to the boolean algebra
    of finite and cofinite subsets of $\mathbb{N}$.}
  Let $\mathcal{O}_X(U)=\{f:U\to \mathbb{R} ~\text{continuous}~\}$
  and take the global section defined by $s(n)=\nicefrac{1}{n}$ and $s(\infty)=0$.
  Then the invertibility open
  $\Support{X}{s}=\{x\in\mathbb{N}\cup\{\infty\}~\vert~s(x)\neq 0\}=\mathbb{N}$ is not compact.
  Hence we can not describe this locally ringed coherent space as a locally ringed lattice.
\end{example}

In the lingo of Grothendieck's ``EGA 1''~\cite{EGA1},
coherent spaces are precisely the qcqs-spaces which
are sober. The technical notion of sobriety ensures that the points of a space can be recovered
from the locale of opens and is crucial to make the representation theorem work.
However, as schemes are always sober, qcqs-schemes are precisely the schemes with a coherent
topology. Now, for a qcqs-scheme $X$, any quasi-compact open $U\subseteq X$ has a finite affine cover
$U=\bigcup_{i=1}^n\mathsf{Spec}(R_i)$. It follows that for any $s\in\mathcal{O}_X(U)$,
its invertibility open can be computed as
$\Support{U}{s}=\bigcup_{i=1}^n D\big(\rest{s}{\mathsf{Spec}(R_i)}\big)$
and is thus quasi-compact.

Moreover, any morphism of schemes between two qcqs-schemes $X\to Y$ is
coherent (quasi-compact). The key result \cite[Lemma 15]{ConstrSchemes}
\begin{center}
  \emph{If $X$ is a qc- and $Y$ a qs-scheme, then any morphism $X\to Y$ is quasi-compact.}
\end{center}
is proved by combining several results from \cite[Ch. 6]{EGA1}.\footnote{In ``The Rising Sea'' this is proved using the so-called ``cancellation theorem'' \cite{RisingSea}.}
This shows that the category of locally ringed lattices contains qcqs-schemes as a full
subcategory and that \cref{def:qcqsSch} gives an equivalent description of this subcategory.

\section{The functor of points approach}\label{sec:fopApproach}
For algebraic geometers it is often convenient to identify a scheme
with its ``functor of points''.\footnote{see e.g.\ \cite[Ch.\ VI]{EisenbudHarris}
  or \cite[Ch.\ II \S 6]{MumfordRedBook}.}
Given a scheme $X$, we can look at the
presheaf $\mathsf{Sch}(\_,X):\mathsf{Sch}^{op}\to\mathsf{Set}$ and by
the Yoneda lemma, two schemes $X$ and $Y$ are isomorphic as schemes if
and only if the presheaves $\mathsf{Sch}(\_,X)$ and $\mathsf{Sch}(\_,Y)$ are
naturally isomorphic. Since affine schemes form a dense
subcategory of the category of schemes, $X$ and $Y$ are isomorphic
schemes if the restrictions
$\rest{\mathsf{Sch}(\_,X)}{\mathsf{Aff}^{op}}$ and
$\rest{\mathsf{Sch}(\_,Y)}{\mathsf{Aff}^{op}}$ are isomorphic as
presheaves. We can push this even further by using that the functor
$\mathsf{Spec}$ is an equivalence of categories between $\mathsf{Aff}$
and $\mathsf{CommRing}^{op}$.  For any scheme $X$ we can define its
functor of points
$h_X:\equiv\mathsf{Sch}(\Spec{(\_)},X):\mathsf{CommRing}\to\mathsf{Set}$
(a presheaf on $\mathsf{CommRing}^{op}$) and two schemes are isomorphic
if and only if their functors of points are naturally isomorphic.
We can do this for our definition of qcqs-schemes as locally ringed lattices
where a morphism of schemes is just a morphism of locally ringed lattices
or equivalently a locally affine morphism of ringed lattices.
We will study the functor of points for locally ringed lattices
in \cref{subsec:functorOfPoints}.

In this section we want  to answer the following question:
Is it possible to characterize the image of the functor
$h:\mathsf{qcqsSch}\hookrightarrow\mathsf{Psh}(\mathsf{CommRing}^{op})$
directly, using purely algebraic and categorical methods, while staying
constructive and predicative?
In the classical literature, this is the so-called
\emph{functor of points approach}, which allows one to
introduce schemes without having to introduce locally ringed spaces
first. Affine schemes are readily accounted for in this approach, as
they correspond to the \emph{representable} presheaves given by the
Yoneda embedding, which we will henceforth denote by
$\mathsf{Sp}:\mathsf{CommRing}^{op}\hookrightarrow\mathsf{Psh}(\mathsf{CommRing}^{op})$.
Indeed, one can immediately verify that $\mathsf{Sp}(A)\cong h_{\Spec{(A)}}$.
Functorial schemes turn out to be those
presheaves $X:\mathsf{CommRing}\to\mathsf{Set}$ that satisfy a certain
locality condition and have an \emph{open cover} by affine
sub-presheaves $\mathsf{Sp}(A)\hookrightarrow X$.

The standard reference developing the basics of algebraic geometry
starting from functors $\mathsf{CommRing}\to\mathsf{Set}$ is
``Introduction to Algebraic Geometry and Algebraic Groups''
by Demazure and Gabriel \cite{DemazureGabriel}.
The central notion here is that of an \emph{open subfunctor}.
These are in a certain sense classified by the
locale of ``Zariski opens''.\footnote{See e.g.\ \cite{MadoreDefScheme}. This fact is
  usually not stated explicitly in textbooks.}
Focusing instead on the notion of \emph{compact open subfunctors},
which are classified by the Zariski lattice,
we can obtain a constructive functorial approach
to qcqs-schemes. In this approach we have to be careful about size issues,
so we will annotate categories with the universe level at which the type
of objects live. Fixing a base level $\ell$,
$\mathsf{CommRing}_\ell$ will be the type or category of ``small'' commutative rings,
whose carrier type lives in $\mathsf{Type}_\ell$. Big commutative rings
will be denoted by $\mathsf{CommRing}_{\ell+1}$ and the same applies to
lattices, locally ringed lattices and the like.

Since it only requires standard tools from category theory and
algebra, this approach lends itself to
a concise definition of qcqs-schemes that is convenient to formalize.
In particular, many of the definitions and results given in this section
are formalized using the \texttt{Cubical Agda} proof assistant
and are already described in \cite{ZeunerHutzler24}, which adapts \cite{DemazureGabriel}
to HoTT/UF while staying constructive and predicative.
We repeat definitions and results from \cite{ZeunerHutzler24}
here for convenience and for the sake of self-containedness
but omit proofs.

\begin{definition}\label{def:ZFunctors}
  The category of $\Z$-functors, denoted $\ZFunctor_\ell$,
  is the category of functors from $\mathsf{CommRing}_\ell$ to $\mathsf{Set}_\ell$.
  We write $\mathsf{Sp}:\mathsf{CommRing}_\ell^{op}\to\ZFunctor_\ell$
  for the Yoneda embedding and $\Aone :\ZFunctor_\ell$ for the
  forgetful functor from commutative rings to sets.
  We say that $X:\ZFunctor_\ell$ is an \emph{affine scheme}
  if there exists a $R:\mathsf{CommRing}_\ell$ with a natural isomorphism
  $X\cong\mathsf{Sp}(R)$.\footnote{As the category of $\Z$-functors is univalent
    and $\mathsf{Sp}$ is fully faithful,
    it is not necessary to use mere existence for defining affine schemes,
    as the type of rings $R$ with $X\cong\mathsf{Sp}(R)$ is a proposition.}
\end{definition}

\begin{example}\label{ex:AoneGm}
  $\Aone$ is an affine scheme, as $\Aone\cong\mathsf{Sp}(\Z[x])$.
  The $\Z$-functor $\mathbb{G}_m$ that sends a ring to its set of units,
  i.e.\ $\mathbb{G}_m(R):\equiv R^\times$, is an affine scheme as
  $\mathbb{G}_m\cong\mathsf{Sp}(\locEl{\Z[x]}{x})$.
\end{example}

\begin{definition}\label{def:globalSectionsFun}
  Let $X:\ZFunctor_\ell$, the \emph{ring of functions}
  $\Ofun(X)$ is the type of natural transformations $\NatTrans{X}{\Aone}$
  equipped with the canonical point-wise operations,
  i.e.\ for $R:\mathsf{CommRing}_\ell$ and $x:X(R)$,
  zero, one, addition and multiplication
  are given by
  \begin{align*}
    &0_R(x) :\equiv 0, \quad  1_R(x):\equiv 1 \\
    &(\alpha+\beta)_R(x) ~:\equiv~ \alpha_R(x) + \beta_R(x) \\
    &(\alpha\cdot\beta)_R(x) ~:\equiv~ \alpha_R(x) \cdot \beta_R(x)
  \end{align*}
  This defines a functor $\Ofun:\ZFunctor_\ell\to\mathsf{CommRing}_{\ell+1}^{op}$,
  whose action on morphisms (natural transformations) is given by precomposition.
\end{definition}

Because the ring of functions of a $\Z$-functor is lives in the successor universe,
we cannot have an adjunction between $\mathsf{Sp}$ and $\mathcal{O}$.
However, the two functors form what is called a \emph{relative coadjunction}
with respect to the functor induced by universe lifting
$\mathsf{lift}:\mathsf{CommRing}^{op}_\ell\to\mathsf{CommRing}^{op}_{\ell+1}$.

\begin{definition}[\cite{nlab:relative_adjoint_functor}]\label{def:relCoadjunction}
  Let $\mathcal{B},\mathcal{C},\mathcal{D}$ be categories with functors
  $l:\mathcal{C}\to\mathcal{D}$, $F:\mathcal{B}\to\mathcal{D}$ and $G:\mathcal{C}\to\mathcal{B}$.
  We say that
  $F$ and $G$ are $l$-relative adjoint, written as $G \prescript{}{l}{\dashv}~F$, if there is a natural
  family of equivalences
  \begin{align*}
    (c:\mathcal{C})~(b:\mathcal{B}) \to \catHom{\mathcal{B}}{G(c)}{b}~\simeq~\catHom{\mathcal{D}}{l(c)}{F(b)}
  \end{align*}
  Dually, we say that
  $F$ and $G$ are $l$-relative coadjoint, written as $F\dashv_{\,l} G$, if there is a natural
  family of equivalences
  \begin{align*}
    (c:\mathcal{C})~(b:\mathcal{B}) \to \catHom{\mathcal{D}}{F(b)}{l(c)}~\simeq~\catHom{\mathcal{B}}{b}{G(c)}
  \end{align*}
\end{definition}

\begin{remark}
  In the setting of \cref{def:relCoadjunction}, a relative adjunction
  $G \prescript{}{l}{\dashv}~F$ induces a relative unit, i.e.\
  for $c:\mathcal{C}$ a map $\eta_c:\catHom{\mathcal{D}}{l(c)}{F(G(c))}$
  natural in $c$.
  Analogously, a relative
  coadjunction $F\dashv_{\,l} G$ induces a relative counit, i.e.\
  for $c:\mathcal{C}$ a map $\varepsilon_c:\catHom{\mathcal{D}}{F(G(c))}{l(c)}$
  natural in $c$. Note that we do not get the other direction in either case.
\end{remark}

\begin{proposition}\label{prop:OSpRelCoadj}
  We have a relative coadjunction
  $\mathcal{O}\dashv_{\mathsf{lift}}\mathsf{Sp}$. In particular,
  for $R:\mathsf{CommRing}_\ell$ and $X:\ZFunctor_\ell$ there is an equivalence
  of types
  \begin{align*}
    {\mathsf{Hom}\big(R,\Ofun(X)\big)}\simeq\big(\NatTrans{X}{\mathsf{Sp}(R)}\big)
  \end{align*}
  which is natural in both $R$ and $X$.
  For any $R:\mathsf{CommRing}_\ell$, the relative counit
  $\varepsilon_R:\mathsf{Hom}\big(R,\Ofun(\mathsf{Sp}(R))\big)$,
  which is obtained by applying the inverse of above isomorphism
  to the identity $\NatTrans{\mathsf{Sp}(R)}{\mathsf{Sp}(R)}$,
  is an isomorphisms of rings.
\end{proposition}

\begin{remark}\label{rem:OSpRelCoadj}
  \cref{prop:OSpRelCoadj} type-checks because the type of ring homomorphisms
  is universe polymorphic, meaning that it can take rings living in different
  universes as arguments. The same holds for the type of isomorphisms/equivalences
  between two types. Since universe lifting commutes with type formers, we get
  \begin{align*}
    {\mathsf{Hom}\big(R,\Ofun(X)\big)}\simeq\catHom{\mathsf{CommRing}^{op}_{\ell+1}}{\Ofun(X)}{\mathsf{lift}(R)}
  \end{align*}
  which is natural in $R$ and $X$. It is modulo this equivalence that
  we actually get the relative coadjunction $\mathcal{O}\dashv_{\mathsf{lift}}\mathsf{Sp}$
  in \cref{prop:OSpRelCoadj}.
\end{remark}

\subsection{Zariski coverage} %%%%%%%%%%%%%%%%%%%%%%%%%%%%%%%%%%%%%%%%%%%%%%%%%%%%%%%%%%%

The category of $\Z$-functors, contains a lot more than the functorial
schemes we want to define. For example, for a qcqs-scheme $X$, we also have a functor
mapping a ring $A$ to the set $\catHom{\mathsf{RDL}^{op}}{\mathsf{Spec}(A)}{X}$
of morphisms of \emph{ringed lattices} from $X$ to $\mathsf{Spec}(A)$.
This functor is of course of little interest to algebraic geometers
and we need to set it apart from the functor of points
$\catHom{\mathsf{LRDL}^{op}}{\mathsf{Spec}(\_)}{X}$
by introducing a ``locality'' condition on $\Z$-functors.
This condition is induced by a certain coverage or Grothendieck topology
on $\mathsf{CommRing}_\ell^{op}$, called the Zariski coverage or topology.
Local $\Z$-functors turn out to be Zariski sheaves.

The notion of Grothendieck topology or coverage generalizes point-set topologies
to arbitrary categories.  It associates to each object in a category
$X:\mathcal{C}$ a family $\mathsf{Cov}(X)$ of covers, where a cover
$\big\{f_i:\catHom{\mathcal{C}}{U_i}{X}\big\}_{i:I}$ is a family of
maps into $X$, i.e.\ a family of objects in the slice category $\nicefrac{\mathcal{C}}{X}$.
These $\mathsf{Cov}(X)$ should satisfy certain closure
properties. If $\mathcal{C}$ has pullbacks,
$\mathsf{Cov}$ should be closed under pullbacks,
meaning that for $\big\{f_i:\catHom{\mathcal{C}}{U_i}{Y}\big\}_{i:I}$ in
$\mathsf{Cov}(Y)$ and $g:\catHom{\mathcal{C}}{X}{Y}$, the family
$\big\{g^*f_i:\catHom{\mathcal{C}}{X\times_Y U_i}{X}\big\}_{i:I}$
should be in $\mathsf{Cov}(X)$. A presheaf
$\mathcal{F}:\mathcal{C}^{op}\to\mathsf{Set}$ can then be defined to be a
sheaf if for any $\big\{f_i:\catHom{\mathcal{C}}{U_i}{X}\big\}_{i:I}$
in $\mathsf{Cov}(X)$ we get an equalizer diagram
\begin{align}\label{eq:sheafProp}
  \mathcal{F}(X) \to \prod_{i:I}\mathcal{F}(U_i)\rightrightarrows
                                              \prod_{i,j:I}\mathcal{F}(U_i \times_X U_j)
\end{align}
In the case where $\mathcal{C}$ is the poset of open subsets of a
topological space $X$, covers are defined the usual way: A family
$(U_i\subseteq U)_{i: I}$ is in $\mathsf{Cov}(U)$ if $\bigcup_i U_i = U$.
Pullbacks are just given by intersection and we recover the usual
topological notion of sheaf.

Note that, while coverages are generally defined as pullback-stable families of covers,
Grothendieck topologies require more closure properties
and their covering families are usually required to be so-called sieves.
Moreover, coverages can be defined
on categories without pullbacks with a slightly weaker condition of ``pullback stability''.
Accordingly, the definition of sheaf with respect to such a generalized coverage
is rephrased in a pullback-free way. For an overview of coverages with and without
pullbacks and Grothendieck topologies see \cite{nlab:coverage}.
For a coverage (with or without pullbacks) one can always find a unique
Grothendieck topology inducing the same notion of sheaf, by taking
``closures'' of covers in a certain sense.\footnote{See e.g.\ \cite[Prop.\ C2.1.9]{JohnstoneSketchesVolII}.}
Still, it is often convenient to work with concrete covers and not their generated sieves,
as they give a simpler description of sheaves.

The formalization of functorial qcqs-schemes described in
\cite{ZeunerHutzler24} uses coverages without pullbacks, even though
$\mathsf{CommRing}^{op}_\ell$ has pullbacks (given by the tensor
product of rings). This does seem to make certain things easier to
formalize, as it requires less categorical machinery.  For the sake of
conceptually simplicity, however, we want to think of the specific
coverage that we are interested in, the Zariski coverage, as pullback-stable
families of covers on $\mathsf{CommRing}_\ell^{op}$ with the
corresponding notion of Zariski sheaf given as in
(\ref{eq:sheafProp}). Consequently, results and proofs regarding the
Zariski coverage are perhaps not as straightforwardly formalizable and
may require careful rephrasing. Ultimately, this should always be possible
as both coverages (with or without pullbacks) generate the same Grothendieck topology
and thus give the same notion of Zariski sheaf.

\begin{definition}\label{def: ZariskiCoverage}
  For $R:\mathsf{CommRing}_\ell^{op}$, a \emph{Zariski cover}
  is indexed by finitely many elements $f_1,...,f_n:R$ such that
  $1\in\langle f_1,...,f_n\rangle$ and assigns to each $f_i$ the
  canonical morphism
  $\nicefrac{\_}{1}:\catHom{\mathsf{CommRing}_\ell^{op}}{\locEl{R}{f_i}}{R}$.
  This defines a coverage, called the \emph{Zariski coverage},
  as \cref{lem: pullbackStability} shows.
  A $\Z$-functor is called \emph{local} if it is a sheaf
  with respect to the Zariski coverage.
\end{definition}

\begin{lemma}\label{lem: pullbackStability}
  Zariski covers are stable under pullbacks.
\end{lemma}

\begin{proof}
  Let $R, A:\mathsf{CommRing}_\ell$ and  $\varphi:\mathsf{Hom}(R,A)$.
  For $f_1,...,f_n:R$ with $1\in\langle f_1,...,f_n\rangle$,
  we get that $1\in\langle \varphi(f_1),...,\varphi(f_n)\rangle$
  and this cover is indeed the pullback (in $\mathsf{CommRing}_\ell^{op}$)
  of the Zariski cover on $R$, as we have a canonical isomorphism
  $\locEl{A}{\varphi(f_i)}\cong A\otimes_R\locEl{R}{f_i}$.
\end{proof}

\begin{theorem}\label{thm:subcanonicity}
  The Zariski coverage is subcanonical, meaning that $\mathsf{Sp}(A)$ is local
  for $A:\mathsf{CommRing}_\ell$.
\end{theorem}

\begin{proof}
  This is essentially the standard algebraic fact that is used to prove that the
  structure sheaf of an affine sheaf satisfies the sheaf property \cite[Prop.\ I-18]{EisenbudHarris}:
  For $f_1,...,f_n:R$ with $1\in\langle f_1,...,f_n\rangle$ we have an equalizer diagram
  \begin{align*}
    R\to\prod_{i=1}^n\locEl{R}{f_i}\rightrightarrows
                                              \prod_{i,j}\locEl{R}{f_if_j}
  \end{align*}
  Since $\locEl{R}{f_if_j}\cong \locEl{R}{f_i}\otimes_R\locEl{R}{f_j}$,
  i.e.\ $\locEl{R}{f_if_j}$ is the pullback in $\mathsf{CommRing}_\ell^{op}$,
  and $\mathsf{Hom}(A,\_)$ is left exact, the sheaf property (\ref{eq:sheafProp}) follows.
  For a formal proof for the Zariski coverage defined without using pullbacks,
  see \cite[Thm.\ 12]{ZeunerHutzler24}.
\end{proof}

\subsection{Compact opens} %%%%%%%%%%%%%%%%%%%%%%%%%%%%%%%%%%%%%%%%%%%%%%%%%%%%%%%%%%%%%%%%%%%%%%%%

The main idea that allowed for a formalization of
functorial qcqs-schemes in \cite{ZeunerHutzler24}
was a streamlined definition of compact open subfunctors
of a $\Z$-functor.
The compact open subfunctors are investigated extensively in
the thesis of Blechschmidt \cite{BlechschmidtPhD}.
They play a prominent role in the development of synthetic algebraic geometry,
i.e.\ algebraic geometry done internally in the topos of local $\Z$-functors,
in the form of compact open propositions.
Openness is usually defined as a property of a subfunctor $U\hookrightarrow X$.
Recall that a subfunctor of a $\Z$-functor $X$ is a $\Z$-functor $U$
with a natural transformation $\NatTrans{U}{X}$ that is pointwise injective.
The property of being an open subfunctor
can directly be restricted to compact openness as in \cite[Def.\ 19.15]{BlechschmidtPhD}.

\begin{definition}\label{def:isAffineCompOpenSubfun}
  A subfunctor
  $U\hookrightarrow\mathsf{Sp}(A)$ is an \emph{affine open} subfunctor
  if there merely exists an ideal $I\subseteq A$
  such that for any ring $B$ we have
  \begin{align}\label{al:isOpenSubfun}
    U(B)\;\cong\;\big\{\varphi:\mathsf{Hom}(A,B)~\vert~\varphi^*I=B\big\}
  \end{align}
  naturally and commuting with the embeddings into $\mathsf{Sp}(A)$.
  Here $\varphi^*I$ is the ideal generated by the image of $I$ under $\varphi$.
  We say that $U\hookrightarrow\mathsf{Sp}(A)$ is an \emph{affine compact open} subfunctor
  if there merely exists a \emph{finitely generated} ideal $I\subseteq A$
  such that we have an isomorphism of subobjects as in (\ref{al:isOpenSubfun}).
\end{definition}

\begin{remark}\label{rem:affCompOpenZarLat}
  Note that the ideal $I$ is not uniquely determined for
  an affine (compact) open subfunctor $U\hookrightarrow\mathsf{Sp}(A)$.
  For ideals $I,J$ of $A$ we have that $\sqrt{I}=\sqrt{J}$ if and
  only if for all rings $B$ and $\varphi:\mathsf{Hom}(A,B)$ we have
  \begin{align*}
    \varphi^*I=B ~\Leftrightarrow~ \varphi^*J=B
  \end{align*}
  If $I=\langle a_1,...,a_n\rangle$ is finitely generated, $U$ is isomorphic (as subobjects) to the
  affine compact open subfunctor given by
  \begin{align*}
    B\;\mapsto\;\;\big\{\varphi:\mathsf{Hom}(A,B)~\vert~1\in\langle\varphi(a_1),\dots,\varphi(a_n)\rangle\,\big\}
  \end{align*}
  By the above argument it follows that this subfunctor is really determined
  by the  element $D(a_1,\dots,a_n):\mathcal{L}_A$, since for
  $J=\langle b_1,\dots,b_m\rangle$ we have
  \begin{align*}
    D(a_1,\dots,a_n)=D(b_1,\dots,b_m) \;\Leftrightarrow\;\sqrt{\langle a_1,\dots,a_n\rangle}=\sqrt{\langle b_1,\dots,b_m\rangle}
  \end{align*}
  We can thus write the corresponding subfunctor as
  \begin{align*}
    B\;\mapsto\;\;\big\{\varphi:\mathsf{Hom}(A,B)~\vert~D(1)=D(\varphi(a_1),...,\varphi(a_n))\,\big\}
  \end{align*}
  In other words the compact open subfunctors of  the affine scheme $\mathsf{Sp}(A)$ are in
  bijection with elements of $\ZarLat{A}$.
\end{remark}

\begin{definition}\label{def:isCompOpenSubfun}
  For $\Z$-functors $U\hookrightarrow X$, we say
  that $U$ is a \emph{(compact) open} subfunctor of $X$, if for any ring $A$ and $A$-valued
  point of $X$, i.e.\ natural transformation $\phi:\NatTrans{\mathsf{Sp}(A)}{X}$, the
  pullback $V:\equiv \mathsf{Sp}(A)\times_X U$ is an affine (compact) subfunctor:
  \begin{center}
    \begin{tikzcd}
      V\ar[r]\ar[d,hook']\pbsign{dr} &U\ar[d,hook'] \\
      \mathsf{Sp}(A)\ar[r,"\phi"] &X
    \end{tikzcd}
  \end{center}
\end{definition}

\begin{remark}
  One easily checks that on $\mathsf{Sp}(A)$ compact open subfunctors and affine compact open
  subfunctors coincide.
\end{remark}

One usually wants to identify subobjects up to natural isomorphism commuting with embeddings.
This complicates things substantially in a formal setting.
The key observation in \cite{ZeunerHutzler24}, allowing to circumvent this issue, is that
the compact open subfunctors are classified by the Zariski lattice.

\begin{definition}\label{def: ZarLatFun}
  Let $\mathcal{L}:\ZFunctor_\ell$ be the $\Z$-functor mapping
  $R:\mathsf{CommRing}_\ell$ to the underlying set of the Zariski lattice
  $\ZL$. The action on morphisms is induced by the universal property of the
  Zariski lattice, i.e.\ for $\varphi:\mathsf{Hom}(A,B)$ we take
  \[
  \begin{tikzcd}
    & A \arrow[dl,"D"']\arrow[dr,"D \circ\varphi"] & \\
    \mathcal{L}_A \arrow[rr,dashed, "\exists!~\varphi^{\mathcal{L}}"'] && \mathcal{L}_B
  \end{tikzcd}
  \]
  The transformation $D(1):\NatTrans{\mathbf{1}}{\mathcal{L}}$ from the terminal
  $\Z$-functor $\mathbf{1}(\_):\equiv\{\ast\}$ to $\mathcal{L}$ is the
  ``pointwise constant'' map $\ast\mapsto D(1)$.
\end{definition}

\begin{proposition}
  For each presheaf with compact open subfunctor $U\hookrightarrow X$ there
  exists a unique natural transformation $\chi:\NatTrans{X}{\mathcal{L}}$ such that the following
  is a pullback square
  \begin{center}
    \begin{tikzcd}
      U\ar[r]\ar[d,hook']\pbsign{dr} &\mathbf{1}\ar[d,"D(1)"] \\
      X\ar[r,"\exists!\chi"'] &\mathcal{L}
    \end{tikzcd}
  \end{center}
\end{proposition}

\begin{proof}
  For a ring $A$ and an $A$-valued point of $X$,
  which modulo Yoneda we can write as  $\phi:\NatTrans{\mathsf{Sp}(A)}{X}$,
  we can pull back $U$ along $\phi$ to obtain an affine compact open subfunctor corresponding
  to a unique $u:\mathcal{L}_A$ and set $\chi(\phi):\equiv u$. This is well-defined
  by \cref{rem:affCompOpenZarLat} and easily seen to be natural in $A$.
  It also follows directly that this gives a pullback square.
\end{proof}

\begin{definition}\label{def: CompOpens}
  Let $X:\ZFunctor_\ell$, a \emph{compact open} of $X$ is a natural transformation
  $U:\NatTrans{X}{\mathcal{L}}$. The \emph{realization} $\coBrackets{U}:\ZFunctor_\ell$
  of a compact open $U$ of $X$, is given by
  \begin{normalfont}
  \begin{align*}
    \coBrackets{U}\,(R)~=~\tySigmaNoParen{x}{X(R)}{\tyPath{U(x)}{D(1)}}
  \end{align*}
  \end{normalfont}
  A compact open $U$ is called \emph{affine}, if its realization is affine, i.e.\
  if there merely exists $R:\mathsf{CommRing}_\ell$ such that
  $\coBrackets{U}\cong\mathsf{Sp}(R)$.
\end{definition}

\begin{remark}
  Indeed, one can check that
  \begin{center}
    \begin{tikzcd}
      \coBrackets{U}\ar[r]\ar[d,hook']\pbsign{dr} &\mathbf{1}\ar[d,"D(1)"] \\
      X\ar[r,"U"'] &\mathcal{L}
    \end{tikzcd}
  \end{center}
  Working with \cref{def: CompOpens} instead of \cref{def:isCompOpenSubfun},
  can be seen as an instance of a common trick in constructive mathematics and
  type theory in particular. In order to avoid setoid or transport hell,
  when working with objects up to some non-trivial equivalence relation, one can instead
  use a suitable index set which is in bijection with the equivalence classes,
  and where each ``index'' corresponds to a canonical member of the corresponding equivalence class.
\end{remark}

\begin{definition}\label{def: compOpenDistLat}
  Let $X:\ZFunctor_\ell$, the \emph{lattice of compact opens}
  $\mathsf{CompOpen}(X)$ is the type $\NatTrans{X}{\mathcal{L}}$
  equipped with the canonical point-wise operations,
  i.e.\ for $R:\mathsf{CommRing}_\ell$ and $x:X(R)$, top, bottom, join and meet are given by
  \begin{align*}
    &1_R(x) = D(1), \quad  0_R(x) = D(0) \\
    &(U\wedge V)_R(x) ~=~ U_R(x) \wedge V_R(x) \\
    &(U\vee V)_R(x) ~=~ U_R(x) \vee V_R(x)
  \end{align*}
  This defines a functor $\mathsf{CompOpen}:\ZFunctor_\ell\to\mathsf{DistLattice}_{\ell+1}^{op}$.
\end{definition}

\begin{definition}\label{def:compOpenCover}
    Let $X:\ZFunctor_\ell$ and $U_1,...,U_n:\NatTrans{X}{\mathcal{L}}$.
    We say that the  $U_i$ form a \emph{compact open cover} of $X$, if
    $\tyPath{1}{\bigvee_{i=1}^n U_i}$ in the lattice $\mathsf{CompOpen}(X)$.
\end{definition}

\begin{remark}
  Being able to define the notion of cover using the lattice structure
  on compact opens is one of the main advantages of using the internal
  Zariski lattice as a classifier.  In the literature, the classically
  valid fact that $U_1,...,U_n:\NatTrans{X}{\mathcal{L}}$ cover $X$ if
  and only if $\bigcup_{i=1}^n\coBrackets{U_i}(k)=X(k)$ for all fields
  $k$, is sometimes used as a definition for the notion of cover
  \cite[Thm.\ VI-14]{EisenbudHarris}.  As the notion of field is
  ambiguous from the constructive point of view, such a definition
  does not appear to be suitable for our purposes. The alternative of
  pulling back to affine compact opens and defining covers in terms of
  ideal addition does work, but \cref{def:compOpenCover} streamlines
  this by making the occurrence of ideal addition implicit in the
  definition of the join operation on the Zariski lattice.  This lets
  us define functorial qcqs-schemes.
\end{remark}

\begin{definition}\label{def: qcqsScheme}
  $X:\ZFunctor_\ell$ is a \emph{qcqs-scheme} if it is a local $\Z$-functor and merely has
  an affine compact open cover, i.e.\ a compact open cover
  $U_1,...,U_n:\NatTrans{X}{\mathcal{L}}$ such that each $U_i$ is affine.
\end{definition}

As an immediate sanity check we get that affine schemes are qcqs-schemes:

\begin{proposition}\label{prop: affineIsQcqsScheme}
  $\mathsf{Sp}(R)$ is a qcqs-scheme, for $R:\mathsf{CommRing}_\ell$.
\end{proposition}
\begin{proof}
\cite[Prop.\ 17]{ZeunerHutzler24}
\end{proof}

For locally ringed lattices we saw that qcqs-schemes are
obtained by ``gluing'' together affine schemes in the sense
that they are a colimit of affine subschemes of a certain shape.
This holds for functorial qcqs-schemes as well, but only in the big
Zariski topos, i.e.\ the category of local $\Z$-functors.
The following lemma is taken from Nieper-Wi{\ss}kirchen's lecture notes
\cite{NWNotes} and adapted to our setting.

\begin{lemma}\label{lem:funMorphismGluing}
  Let $X$ be a $\Z$-functor and $U_1,...,U_n:\NatTrans{X}{\mathcal{L}}$
  be a compact open cover of $X$. Let $Y$ be a \emph{local} $\Z$-functor
  with morphisms $\alpha_i:\NatTrans{\coBrackets{U_i}}{Y}$ for $i$ in
  $1,...,n$ such that for every $i,j$ the following square commutes
  \begin{center}
    \begin{tikzcd}
      \coBrackets{U_i\wedge U_j}\ar[r,hook]\ar[d,hook'] &\coBrackets{U_i}\ar[d,"\alpha_i"] \\
      \coBrackets{U_j}\ar[r,"\alpha_j"'] &Y
    \end{tikzcd}
  \end{center}
  Then there exists a unique $\alpha:\NatTrans{X}{Y}$ such that
  for every $i$ in $1,...,n$ the following triangle commutes
  \[
  \begin{tikzcd}
    & X \arrow[dr,"\alpha"] & \\
    \coBrackets{U_i} \arrow[ur,hook]\arrow[rr, "\alpha_i"'] && Y
  \end{tikzcd}
  \]
\end{lemma}

\begin{proof}
  First, we treat the case where $X=\mathsf{Sp}(R)$ is affine.
  Without loss of generality, we can assume that for every
  $i$ in $1,...,n$, $U_i=D(f_i)$ is a basic open.
  Otherwise, we may use that, by Yoneda, we have a cover by basic opens
  $U_i=\bigvee_k D(f_{ik})$ and refine our cover accordingly. Note that
  the assumption of the lemma still holds for the refined cover.
  However, in the case of a cover by basic opens,
  $\alpha_i$ corresponds to an element of $Y(\locEl{R}{f_i})$
  by Yoneda, and the lemma follows as $Y$ was assumed local.

  For the general case let $R:\mathsf{CommRing}_\ell$ and $x:X(R)$.
  By Yoneda we can regard the $U_i(x):\ZL$ as a compact open cover of $\mathsf{Sp}(R)$.
  The $\alpha_i$ induce natural transformations $\NatTrans{\coBrackets{U_i(x)}}{Y}$,
  given on $A$-valued points by the function
  \begin{align*}
    \coBrackets{U_i(x)}(A) = \{\,\varphi:\mathsf{Hom}(R,A)\,\vert\,U_i(\rest{x}{\varphi})=D(1)\,\}~\to~Y(A)
  \end{align*}
  mapping $\varphi$ to $\alpha_i(\rest{x}{\varphi})$, where
  $\rest{x}{\varphi}\,:\equiv X(\varphi)(x)$ is the restriction of  $x$ along $\varphi$.
  We can apply the above argument for the affine case to obtain a unique
  $\NatTrans{\mathsf{Sp}(R)}{Y}$, which gives us the desired $\alpha(x):Y(R)$.
  We omit the proof that this is natural and commutes with the $\alpha_i$.
\end{proof}

\begin{corollary}
  If $X:\ZFunctor_\ell$ is a qcqs-scheme and $U_1,...,U_n:\NatTrans{X}{\mathcal{L}}$
  an (affine) compact open cover of $X$, then
  \begin{align*}
    X~\cong~\mathsf{colim}\,\big\{\coBrackets{U_i}\leftarrow \coBrackets{U_i\wedge U_j}\rightarrow \coBrackets{U_j}\big\}
  \end{align*}
  in the category if local $\Z$-functors, i.e.\ Zariski sheaves.
  In particular, if we have $\coBrackets{U_i}\cong\mathsf{Sp}(A_i)$,
  and affine compact open covers
  $\big(V_{ijk}:\NatTrans{\coBrackets{U_i\wedge U_j}}{\mathcal{L}}\big)_k$
  with $\coBrackets{V_{ijk}}\cong\mathsf{Sp}(A_{ijk})$,
  we have
  \begin{align*}
    X~\cong~\mathsf{colim}\,\big\{\mathsf{Sp}(A_i)\leftarrow \mathsf{Sp}(A_{ijk})\rightarrow \mathsf{Sp}(A_j)\big\}
  \end{align*}
\end{corollary}

\subsection{Open subschemes} %%%%%%%%%%%%%%%%%%%%%%%%%%%%%%%%%%%%%%%%%%%%%%%%%%%%%%%%%%%%%%%%%%%%%%%%%
Let us conclude this section by proving that compact opens of qcqs-schemes are qcqs-schemes.
The special case of compact opens of affine schemes is formalized in \cite{ZeunerHutzler24}.
First, we treat locality.

\begin{lemma}\label{lem:ZarLatLocal}
  $\mathcal{L}$ is local.
\end{lemma}

\begin{proof}
  For $f_1,...,f_n:R$ with $1\in\langle f_1,...,f_n\rangle$, we get $D(1)=D(f_1,...,f_n)$ in $\ZL$.
  By the ``gluing lemma for distributive lattices'' \cite[Lemma 5]{ProjSpec}
  we get an equalizer diagram
  \begin{align*}
    \ZL \to \prod_{i=1}^n\downarrow\! D(f_i)\rightrightarrows
                                              \prod_{i,j}\downarrow\! D(f_if_j)
  \end{align*}
  Now, for $i,j$ in $1,...,n$ let $\psi_i:\ZarLat{\locEl{R}{f_i}}\cong\,\downarrow\! D(f_i)$
  and $\psi_{ij}:\ZarLat{\locEl{R}{f_if_j}}\cong\,\downarrow\! D(f_if_j)$
  be the respective isomorphisms given by \cref{lem:isoLocDownSet}. The following diagram
  commutes
  \begin{center}
  \begin{tikzcd}
    \ZarLat{\locEl{R}{f_i}}\ar[r, "\simeq","\psi_i"']\ar[d] &\downarrow\! D(f_i)\ar[d, "\_\wedge D(f_j)"] \\
    \ZarLat{\locEl{R}{f_if_j}}\ar[r,"\simeq","\psi_{ij}"'] &\downarrow\! D(f_if_j)
  \end{tikzcd}
  \end{center}
  giving us an equalizer diagram
  \begin{align*}
    \ZL \to \prod_{i=1}^n\ZarLat{\locEl{R}{f_i}}\rightrightarrows
                                              \prod_{i,j}\ZarLat{\locEl{R}{f_if_j}}
  \end{align*}
  which finishes the proof.
\end{proof}

\begin{lemma}\label{lem:isLocalCompOpenOfLocal}
  The realization $\coBrackets{U}$ of a compact open $U:\NatTrans{X}{\mathcal{L}}$
  of a local $\Z$-functor $X$ is local.
\end{lemma}

\begin{proof}
  \cite[Lemma 21]{ZeunerHutzler24}. Note that this proof uses the fact
  that $\mathcal{L}$ is \emph{separated} wrt.\ the Zariski coverage.
  This follows from \cref{lem:ZarLatLocal}, as sheaves are always separated.
  Alternatively, see \cite[Lemma 20]{ZeunerHutzler24} for a direct proof.
\end{proof}

It remains to prove that compact opens of qcqs-schemes merely have an affine cover.
Compact opens of affine schemes can be covered by ``basic opens'',
as was shown in \cite{ZeunerHutzler24}.

\begin{definition}\label{def:standardOpen}
  Let $R:\mathsf{CommRing}_\ell$ and $f:R$, the standard open
  $D(f):\NatTrans{\mathsf{Sp}(R)}{\mathcal{L}}$ is given by applying the Yoneda lemma
  to the basic open $D(f):\ZL$.
\end{definition}

\begin{lemma}\label{prop:isAffineStandardOpen}
  For $R:\mathsf{CommRing}_\ell$ and $f:R$, the standard open $D(f)$ is affine.
  In particular one has a natural isomorphism
  ${\coBrackets{D(f)}}\cong{\mathsf{Sp}\big(\locEl{R}{f}\big)}$.
\end{lemma}

\begin{proof}
\cite[Prop.\ 23]{ZeunerHutzler24}
\end{proof}

\begin{lemma}\label{lem:isQcqsSchCompOpenAffine}
  For any ring $R:\mathsf{CommRing}_\ell$, the realization $\coBrackets{U}$ of a compact open
  $U:\NatTrans{\mathsf{Sp}(R)}{\mathcal{L}}$ is a qcqs-scheme.
\end{lemma}

\begin{proof}
\cite[Thm.\ 24]{ZeunerHutzler24}
\end{proof}

Now we want to generalize this to the case of a compact open of an arbitrary qcqs-scheme.
The key result is the following lemma
that will play an important role later on as well.

\begin{lemma}\label{lem:compOpenDownIso}
  For any $\Z$-functor $X$ and $U:\NatTrans{X}{\mathcal{L}}$,
  we have an isomorphism of lattices
  \begin{align*}
    \psi_U: \mathsf{CompOpen}\big(\coBrackets{U}\big)~\cong~\downarrow U
  \end{align*}
  that preserves realizations, i.e.\ $\coBrackets{V}\cong\coBrackets{\psi_U(V)}$
  naturally for $V:\NatTrans{\coBrackets{U}}{\mathcal{L}}$.
\end{lemma}

\begin{proof}
  Constructing the inverse map
  $\psi_U^{-1}:~\downarrow\! U\to \big(\NatTrans{\coBrackets{U}}{\mathcal{L}}\big)$
  is actually the straightforward part. We have a map
  \begin{align*}
    \big(\NatTrans{X}{\mathcal{L}}\big)~\to~\big(\NatTrans{\coBrackets{U}}{\mathcal{L}}\big)
  \end{align*}
  which is given as follows: For $V:\NatTrans{X}{\mathcal{L}}$
  we compose with the canonical inclusion $i_U:\NatTrans{\coBrackets{U}}{X}$,
  i.e.\ take $V\circ i_U$.
  Note that we have $V\circ i_U=(U\wedge V)\circ i_U$.
  We can then take $\psi_U^{-1}$ to be the restriction of this map to
  $\downarrow\! U$, which clearly is a lattice morphism.
  We also get that $\coBrackets{V}\cong\coBrackets{\psi_U^{-1}(V)}$ for $V\leq U$,
  as for any ring $R$ and point $x:X(R)$, $V(x)=D(1)$ implies $U(x)=D(1)$.

  To construct $\psi_U$ let $V:\NatTrans{\coBrackets{U}}{\mathcal{L}}$,
  $R$ a ring and $x:X(R)$. Assume that we have $f_1,...,f_n:R$
  such that $U(x)= D(f_1,...,f_n)$. Let
  $x_i:\equiv X(\nicefrac{\_}{1})(x):X\big(\locEl{R}{f_i}\big)$
  be the restriction of $x$ along the canonical $\nicefrac{\_}{1}:R\to\locEl{R}{f_i}$,
  where $i$ in $1,...,n$. Now, $U(x_i)=D(1)$ and we may apply $V$ to $x_i$, giving us
  $V(x_i):\ZarLat{\locEl{R}{f_i}}$. Using the isomorphisms
  $\psi_{f_i}:\ZarLat{\locEl{R}{f_i}}\cong\,\downarrow\! D(f_i)$
  from \cref{lem:isoLocDownSet}, we get $\psi_{f_i}(V(x_i)):~\downarrow\! D(f_i)$
  for each $i$ in $1,...,n$. We want to set
  \begin{align*}
    \psi_U(V)(x)~:\equiv~\textstyle \bigvee_{i=1}^n\psi_{f_i}(V(x_i)) ~\leq~U(x)
  \end{align*}
  but we need to check that this is well-defined, i.e.\ independent of the choice
  of the $f_i$. So let $g_1,...,g_m:R$ be such that $D(f_1,...,f_n)=D(g_1,...,g_m)=U(x)$.
  For $i$ in $1,...,n$ and $j$ in $1,...,m$, let
  $x_{ij}:X\big(\locEl{R}{f_ig_j}\big)$ be the restriction of $x$ along
  $\nicefrac{\_}{1}:R\to\locEl{R}{f_ig_j}$ and
  $\psi_{f_ig_j}:\ZarLat{\locEl{R}{f_ig_j}}\cong\,\downarrow\! D(f_ig_j)$
  be the corresponding isomorphism given by \cref{lem:isoLocDownSet}.
  For the canonical map $\gamma_{ij}:\locEl{R}{f_i}\to\locEl{R}{f_ig_j}$
  we get that the following diagram commutes
  \begin{center}
  \begin{tikzcd}
    \ZarLat{\locEl{R}{f_i}}\ar[r, "\simeq","\psi_{f_i}"']\ar[d,"\gamma_{ij}^{\mathcal{L}}"'] &\downarrow\! D(f_i)\ar[d, "\_\wedge D(g_j)"] \\
    \ZarLat{\locEl{R}{f_ig_j}}\ar[r,"\simeq","\psi_{f_ig_j}"'] &\downarrow\! D(f_ig_j)
  \end{tikzcd}
  \end{center}
  Moreover, we get $\gamma_{ij}^{\mathcal{L}}(V(x_i))=V(x_{ij})$ by the naturality of $X$ and $V$.
  This gives us
  \begin{align*}
    \psi_{f_i}(V(x_i)) ~=~ \textstyle\bigvee_{j=1}^m \big(\psi_{f_i}(V(x_i))\wedge D(g_j)\big) ~=~\textstyle\bigvee_{j=1}^m \psi_{f_ig_j}(V(x_{ij}))
  \end{align*}
  Repeating the above argument with $i$ and $j$ swapped shows well-definedness:
  \begin{align*}
    \textstyle \bigvee_{i=1}^n\psi_{f_i}(V(x_i))~=~\textstyle \bigvee_{i,j}\psi_{f_ig_j}(V(x_{ij}))~=~\textstyle \bigvee_{j=1}^m\psi_{g_j}(V(x_j))
  \end{align*}
  To show the naturality of this construction, take $\varphi:\mathsf{Hom}(R,A)$ and $x:X(R)$
  with its restriction along $\varphi$ denoted $\rest{x}{\varphi}\,:\equiv X(\varphi)(x)$.
  Since we have to prove a proposition, we can assume that we have $f_1,...,f_n:R$
  such that $U(x)= D(f_1,...,f_n)$ and thus
  $\psi_U(V)(x)~=~\textstyle \bigvee_{i=1}^n\psi_{f_i}(V(x_i))$, with $x_i$ defined
  as above for $i$ in $1,...,n$. It follows that
  \begin{align*}
    \varphi^{\mathcal{L}}\big(\psi_U(V)(x)) ~&=~  \textstyle\bigvee_{i=1}^n\varphi^{\mathcal{L}}(\psi_{f_i}(V(x_i)))\\
    ~&=~  \textstyle\bigvee_{i=1}^n\psi_{f_i}\big((\nicefrac{\varphi}{1})^{\mathcal{L}}(V(x_i))\big) \\
    ~&=~  \textstyle\bigvee_{i=1}^n\psi_{\varphi(f_i)}\big(V(\rest{x}{\nicefrac{\varphi}{1}})\big) \\
    ~&=~\psi_U(V)(\rest{x}{\varphi})
  \end{align*}
  Here $\rest{x}{\nicefrac{\varphi}{1}}$ is the restriction along the canonical
  $\nicefrac{\varphi}{1}:R\to\locEl{A}{\varphi(f_i)}$ for each respective $i$ in $1,...,n$.
  The second equality above holds because the following diagram commutes:
  \begin{center}
  \begin{tikzcd}
    \ZarLat{\locEl{R}{f_i}}\ar[r, "\simeq","\psi_{f_i}"']\ar[d,"(\nicefrac{\varphi}{1})^{\mathcal{L}}"'] &\downarrow\! D(f_i)\ar[d, "\varphi^{\mathcal{L}}"] \\
    \ZarLat{\locEl{A}{\varphi(f_i)}}\ar[r,"\simeq","\psi_{\varphi(f_i)}"'] &\downarrow\! D(\varphi(f_i))
  \end{tikzcd}
  \end{center}
  The last equality holds since $U(\rest{x}{\varphi})=D(\varphi(f_1),..., \varphi(f_n))$.
  This finishes the proof that $\psi_U(V)$ is a natural transformation.
  We omit the proof that $\psi_U$ is inverse to $\psi^{-1}_U$.
\end{proof}

\begin{lemma}\label{lem:affCoverJoin}
  Let $X$ be $\Z$-functor and $U,V:\NatTrans{X}{\mathcal{L}}$ be compact opens.
  If both $\coBrackets{U}$ and $\coBrackets{V}$ merely have an affine cover, then
  $\coBrackets{U\vee V}$ merely has an affine cover.
\end{lemma}

\begin{proof}
  Let $U_1,..,U_n:\NatTrans{\coBrackets{U}}{\mathcal{L}}$
  and $V_1,..,V_m:\NatTrans{\coBrackets{V}}{\mathcal{L}}$
  be affine covers of $\coBrackets{U}$ and $\coBrackets{V}$
  respectively. Using the isomorphisms $\psi_U$ and $\psi_V$
  given by applying \cref{lem:compOpenDownIso} to $U$ and $V$,
  we get compact opens $\psi_U(U_i)$ and $\psi_V(V_j)$ of $X$ that
  are affine and cover $U\vee V$, i.e.\
  \begin{align*}
    \big(\textstyle\bigvee_{i=1}^n\psi_U(U_i)\big) \,\vee\,\big(\bigvee_{j=1}^m\psi_V(V_j)\big)~=~ U\vee V
  \end{align*}
  Using \cref{lem:compOpenDownIso} for $U\vee V$, we get that the
  $\psi_{U\vee V}^{-1}(\psi_U(U_i))$ and
  $\psi_{U\vee V}^{-1}(\psi_V(V_j))$
  together give an affine cover of $\coBrackets{U\vee V}$.
\end{proof}

\begin{theorem}\label{thm:isQcqsSchCompOpen}
  The realization $\coBrackets{U}$ of a compact open
  $U:\NatTrans{X}{\mathcal{L}}$ of a qcqs-scheme $X$ is a qcqs-scheme.
\end{theorem}

\begin{proof}
  $\coBrackets{U}$ is local by \cref{lem:isLocalCompOpenOfLocal}.
  To show that it has an affine cover,
  let $U_1,...,U_n:\NatTrans{X}{\mathcal{L}}$ be an affine cover of $X$.
  For $i$ in $1,...,n$, the compact open $\psi_{U_i}^{-1}(U\wedge U_i)$
  of $\coBrackets{U_i}$ has an affine cover by
  \cref{lem:isQcqsSchCompOpenAffine}.\footnote{For $\coBrackets{U_i}\cong\mathsf{Sp}(A_i)$,
  one may think of $\psi_{U_i}^{-1}(U\wedge U_i)$ as representing the compact open subscheme
  $\coBrackets{U\wedge U_i}\hookrightarrow\mathsf{Sp}(A_i)$.}
  As explained in the proof of \cref{lem:compOpenDownIso}, we get
  natural isomorphisms of $\Z$-functors
  \begin{align*}
    \coBrackets{\psi_{U_i}^{-1}(U\wedge U_i)}~\cong~\coBrackets{U\wedge U_i}~\cong~\coBrackets{\psi_{U}^{-1}(U\wedge U_i)}
  \end{align*}
  which can be made into equalities of $\Z$-functors, as categories of functors
  into a univalent category are always
  univalent. We can transport the property of having an affine cover along these
  equalities, giving us that each $\psi_{U}^{-1}(U\wedge U_i)$
  (a compact open of $\coBrackets{U}$) has an affine cover.
  We can then iterate \cref{lem:affCoverJoin} to obtain an affine cover of $\coBrackets{U}$.
\end{proof}

\section{Equivalence of approaches}\label{sec:ComparisonThm}

Having given two constructive definitions of qcqs-scheme, we want show
that these two notions coincide.  The proof strategy is mostly
analogous to that of the classical \emph{comparison theorem} given by
Demazure and Gabriel \cite[Ch.\ I, \S 1, no 4.4]{DemazureGabriel}.  We
construct an adjunction between $\Z$-functors and the opposite
category of locally ringed lattices.  For proving properties of the
left adjoint that ``realizes'' a $\Z$-functor as a locally ringed
lattice (in particular \cref{prop:isFullyFaithfulGeoRel} and
\cref{prop:relAdjSch}), we take inspiration from the lecture notes of
Nieper-Wi{\ss}kirchen \cite{NWNotes}. The notes often include concrete
computations that can be constructivized even when being explicit
about universe levels.  Technically, the purported adjunction will be
a coadjunction relative to universe lifting, as in
\cref{prop:OSpRelCoadj}. We then show that when restricted to the
respective full subcategories of qcqs-schemes, this becomes an
equivalence of categories.  In order to dispose of relativity with
respect to universe lifting, we use univalence.

\subsection{Functor of points of a locally ringed lattice}\label{subsec:functorOfPoints}

Developing scheme theory starting from $\Z$-functors
is an exception rather than norm in the introductory literature on algebraic
geometry. Most often schemes are defined as locally ringed spaces and in a second step
they are assigned a $\Z$-functor, called the functor of points of a scheme or
locally ringed space.\footnote{See e.g.\ \cite[Ch.\ 6]{EisenbudHarris} for an in-depth discussion in a classical source.}
We can do the same for locally ringed lattices.
In this section we want to define the functor of points for locally ringed lattices
and prove some basic properties for the functor of points of a qcqs-scheme.
In order to avoid confusion, we
will henceforth use qcqs-scheme only for lattice-theoretic qcqs-schemes and speak
of functorial qcqs-schemes otherwise.

\begin{definition}\label{def:functorOfPoints}
  The \emph{functor of points} associated to a locally ringed lattice
  $X:\equiv (L_X,\mathcal{O}_X)$ is the $\Z$-functor given by
  \begin{align*}
    h_X ~:\equiv~\catHom{\mathsf{LRDL}_\ell^{op}}{\mathsf{Spec}(\_)}{X}
  \end{align*}
  This defines a functor
  $\mathsf{LRDL}_\ell^{op}\to \ZFunctor_\ell$
  where the action on moprhisms is given point-wise by post-composition.
\end{definition}

\begin{lemma}\label{lem:functorOfPointsLocal}
  The functor of points $h_X$ of a locally ringed lattice $X$ is local.
\end{lemma}

\begin{proof}
  Let $f_1,...,f_n:R$ be such that $1\in\langle f_1,...,f_n\rangle$.
  Since $\rest{\mathsf{Spec}(R)}{D(f_i)}\cong\mathsf{Spec}(\locEl{R}{f_i})$ as
  locally ringed lattices \cite[Ex.\ 1]{ConstrSchemes}, we get
  by \cref{prop:gluingLatticeQcqsSchemes} that
  in $\mathsf{LRDL}^{op}$
  \begin{align*}
    \mathsf{Spec}(R)~\cong~\mathsf{colim}\,\big\{\mathsf{Spec}(\locEl{R}{f_i})\leftarrow \mathsf{Spec}(\locEl{R}{f_if_j})\rightarrow \mathsf{Spec}(\locEl{R}{f_j})\big\}
  \end{align*}
  Thus the following is an equalizer diagram
  \begin{align*}
    h_X(R)\to\prod_{i=1}^n h_X(\locEl{R}{f_i})\rightrightarrows
                                              \prod_{i,j}h_X(\locEl{R}{f_if_j})
  \end{align*}
  which finishes the proof.
\end{proof}

\begin{lemma}\label{lem:natIsohSpecSp}
  $h_{\mathsf{Spec}(R)}$ is naturally isomorphic to $\mathsf{Sp}(R)$.
\end{lemma}

\begin{proof}
  $\mathsf{Spec}$ is fully faithful.
\end{proof}

\begin{lemma}\label{lem:hLatToCompOpen}
  For a locally ringed lattice $X:\equiv(L_X,\mathcal{O}_X)$,
  we have a lattice homomorphism $\varepsilon^*:L_X\to\mathsf{CompOpen}(h_X)$
  such that for each $u:L_X$ we get
  a natural isomorphism $\coBrackets{\varepsilon^*(u)}\cong h_{\rest{X\,}{u}}$.\footnote{This
    $\varepsilon^*$ turns out to be the lattice morphism corresponding to the relative
    counit induced by \cref{prop:relAdjSch}, hence the name. We will not use this fact
    explicitly, however.}
\end{lemma}

\begin{proof}
  Let $u:L_X$. For a ring $R$ and
  $\pi:\catHom{\mathsf{LRDL}_\ell^{op}}{\mathsf{Spec}(R)}{X}$
  we can define a compact open by setting $\varepsilon^*(u)_R(\pi):\equiv\pi^*(u)$.
  This clearly defines a morphism of lattices.
  Moreover, this gives us
  \begin{align*}
    \coBrackets{\varepsilon^*(u)}(R) ~&\simeq~ \tySigmaNoParen{\pi}{\catHom{\mathsf{LRDL}_\ell^{op}}{\mathsf{Spec}(R)}{X}}{\pi^*(u)=D(1)}  \\[.5em] &\simeq~ \catHom{\mathsf{LRDL}_\ell^{op}}{\mathsf{Spec}(R)}{\rest{X}{u}}
  \end{align*}
  where the second equivalence is given composition with the ``inclusion''
  $(\_\wedge u):\catHom{\mathsf{LRDL}_\ell^{op}}{\rest{X}{u}}{X}$.
  We omit the proof that this is natural in $R$.
\end{proof}

\begin{lemma}\label{lem:isFunSchhSch}
  If $X:\equiv (L_X,\mathcal{O}_X)$ is a qcqs-scheme, then $h_X$ is
  a functorial qcqs-scheme.
\end{lemma}

\begin{proof}
  By \cref{lem:functorOfPointsLocal}, $X$ is local. Assume that
  $u_1,...,u_n:L_X$ is an affine cover with $\rest{X}{u_i}\cong\mathsf{Spec}(R_i)$
  as locally ringed lattices. Then the $\varepsilon^*(u_i)$ cover $h_X$ and are affine as
  \begin{align*}
    \coBrackets{\varepsilon^*(u)}\cong h_{\rest{X\,}{u_i}}\cong h_{\mathsf{Spec}(R_i)}\cong\mathsf{Sp}(R_i)
  \end{align*}
  By \cref{lem:hLatToCompOpen} and \cref{lem:natIsohSpecSp}
\end{proof}

This allows us to regard $h_{(\_)}$ as a functor
$\mathsf{QcQsSch}_\ell\to\mathsf{FunQcQsSch}_\ell$.
This restricted functor of points functor is fully faithful.

\begin{proposition}\label{prop:isFullyFaithfulh}
  The functor $h_{(\_)}:\mathsf{QcQsSch}_\ell\to\mathsf{FunQcQsSch}_\ell$
  is fully faithful.
\end{proposition}

\begin{proof}
  Let $X,Y$ be qcqs-schemes and let $u_1,...,u_n:L_X$ be an affine cover
  with $\rest{X}{u_i}\cong\mathsf{Spec}(R_i)$.
  By \cref{lem:funMorphismGluing}, any $\alpha:\NatTrans{h_X}{h_Y}$
  is uniquely determined by the induced restrictions
  $\alpha_i:\big(\NatTrans{\coBrackets{\varepsilon^*(u_i)}}{h_Y}\big)$.
  Combining the proof of \cref{lem:isFunSchhSch} with Yoneda and our assumptions,
  we get a chain of equivalences
  \begin{align*}
    \big(\NatTrans{\coBrackets{\varepsilon^*(u_i)}}{h_Y}\big)~&\simeq~\big(\NatTrans{\mathsf{Sp}(R_i)}{h_Y}\big) \\
    &\simeq~\catHom{\mathsf{LRDL}_\ell^{op}}{\mathsf{Spec}(R_i)}{Y}~\simeq~\catHom{\mathsf{LRDL}_\ell^{op}}{\rest{X}{u_i}}{Y}
  \end{align*}
  Now, let $\pi_i:\catHom{\mathsf{LRDL}^{op}}{\rest{X}{u_i}}{Y}$ be the morphism
  that we obtain by applying the chain of equivalences to $\alpha_i$.
  By \cref{prop:gluingLatticeQcqsSchemes}, the $\pi_i$
  determine a $\pi:\catHom{\mathsf{LRDL}^{op}}{{X}}{Y}$.
  Cumbersome computations reveal that modulo precomposition with
  the natural isomorphism $\coBrackets{\varepsilon^*(u_i)}\cong h_{\rest{X\,}{u_i}}$,
  given by \cref{lem:hLatToCompOpen}, $h_{\pi_i}$ is just $\alpha_i$.
  Hence, $\pi$ is unique such that $h_\pi=\alpha$.
\end{proof}

\subsection{Geometric realization of a $\Z$-functor} %%%%%%%%%%%%%%%%%%%%%%%%%%%%%%%%%%
We have seen that to a $\Z$-functor $X$ we can assign a (big) distributive lattice.
We can actually make this into a locally ringed lattice that on representables
$\mathsf{Sp}(R)$ is just $\mathsf{Spec}(R)$, modulo universe lifting.

\begin{definition}\label{def:realizationFun}
  For a presheaf $X$, the locally ringed lattice $\abs{X}:\mathsf{LRDL}_{\ell+1}^{op}$
  is the lattice $L_{\abs{X}}:\equiv\mathsf{CompOpen}(X)$ with a sheaf $\mathcal{O}_{\abs{X}}$
  assigning to each $U:\NatTrans{X}{\mathcal{L}}$
  the ring of functions $\NatTrans{\coBrackets{U}}{\mathbb{A}^1}$ on $U$'s realization,
  i.e.\
  \begin{align*}
    \mathcal{O}_{\abs{X}}(U)~:\equiv~\mathcal{O}\big(\coBrackets{U}\big)
  \end{align*}
  The invertibility support on $(\mathsf{CompOpen}(X),\mathcal{O}_{\abs{X}})$
  is given as follows: For a compact open $U:\NatTrans{X}{\mathcal{L}}$ and
  function $s:\NatTrans{\coBrackets{U}}{\mathbb{A}^1}$ we get a natural transformation
  $U_s:\NatTrans{\coBrackets{U}}{\mathcal{L}}$, which for $x:X(R)$ with
  $U(x)=D(1)$ is given by $U_s(x):\equiv D(s(x))$. By \cref{lem:compOpenDownIso}
  this corresponds to a compact open in $\downarrow\! U$, i.e. we can define
  \begin{align*}
    \Support{U}{s}~:\equiv~\psi_U(U_s)
  \end{align*}
  For a natural transformation $\alpha:\NatTrans{X}{Y}$ we get a morphism
  of locally ringed lattices $\abs{\alpha}:\catHom{\mathsf{LRDL}_{\ell+1}^{op}}{\abs{X}}{\abs{Y}}$
  as follows: The lattice morphism $\abs{\alpha}^*$ acts on
  compact opens by precomposition with $\alpha$.
  For $U:\NatTrans{Y}{\mathcal{L}}$, we get a canonical transformation
  $\rest{\alpha}{U}:\NatTrans{\coBrackets{U\circ\alpha}}{\coBrackets{U}}$
  and we can define
  \begin{align*}
    \abs{\alpha}^\sharp_U:\equiv (\_\circ\rest{\alpha}{U}) ~:~\catHom{\mathsf{Hom}}{\mathcal{O}_{\abs{Y}}(U)}{\mathcal{O}_{\abs{X}}(\underbrace{U\circ\alpha}_{\equiv\,\abs{\alpha}^*(U)})}
  \end{align*}
  This defines the desired natural transformation
  $\abs{\alpha}^\sharp:\NatTrans{\mathcal{O}_{\abs{Y}}}{\abs{\alpha}_*\mathcal{O}_{\abs{X}}}$.
  The functor $\abs{\_}:\ZFunctor_\ell\to\mathsf{LRDL}_{\ell+1}^{op}$ is
  called the \emph{geometric realization functor}.
\end{definition}

There are several things to check in order to show that the above is a sensible
definition. Apart from various functoriality, naturality and homomorphism conditions,
we have to check the following four lemmas.

\begin{lemma}
  For any $\Z$-functor $X$, $\mathcal{O}_{\abs{X}}$ is a sheaf on the lattice
  $\mathsf{CompOpen}(X)$.
\end{lemma}

\begin{proof}
  For $U=\bigvee_{i=1}^n U_i$, we get a cover of $\mathsf{CompOpen}\big(\coBrackets{U}\big)$
  by the compact opens $\psi_U^{-1}(U_i)$. Since $\Aone$ is affine and hence local,
  we can use \cref{lem:funMorphismGluing} to see that any function
  $\mathcal{O}\big(\coBrackets{U}\big)$ is uniquely determined by its restriction to the
  $\mathcal{O}\big(\coBrackets{\psi_U^{-1}(U_i)}\big)\cong\mathcal{O}\big(\coBrackets{U_i}\big)$.
  This means that we have an equalizer diagram
  \begin{align*}
    \mathcal{O}\big(\coBrackets{U}\big)\to\prod_{i=1}^n \mathcal{O}\big(\coBrackets{U_i}\big)\rightrightarrows\prod_{i,j}\mathcal{O}\big(\coBrackets{U_i\wedge U_j}\big)
  \end{align*}
  which finishes the proof.
\end{proof}

\begin{lemma}
  For any $\Z$-functor $X$, compact open $U:\NatTrans{X}{\mathcal{L}}$
  and function $s:\NatTrans{\coBrackets{U}}{\mathbb{A}^1}$,
  $\Support{U}{s}$ is an invertibility supremum in $\downarrow\! U$.
\end{lemma}

\begin{proof}
  By \cref{lem:compOpenDownIso}, it suffices to check that $U_s$ is the largest element of
  $\mathsf{CompOpen}\big(\coBrackets{U}\big)$ where (the restriction of) $s$
  becomes invertible. This follows immediately from the fact that,
  for $x:X(R)$ with $U(x)=D(1)$, we have $U_s(x)\equiv D(s(x)) = D(1)$ if and only if
  $s(x)\in R^\times$.
\end{proof}

\begin{lemma}
  For any $\Z$-functor $X$ and compact open $U:\NatTrans{X}{\mathcal{L}}$,
  $\mathcal{D}_U$ is a support.
\end{lemma}

\begin{proof}
  Again by \cref{lem:compOpenDownIso}, it suffices to verify that
  \begin{align*}
    U_{(\_)}:\big(\NatTrans{\coBrackets{U}}{\Aone}\big)\to\big(\NatTrans{\coBrackets{U}}{\mathcal{L}}\big)
  \end{align*}
  is a support. This can be checked point-wise. For example,
  $U_{s+t}\leq U_s\vee U_t$ follows from the fact
  that $D(s(x)+t(x))\leq D(s(x))\vee D(t(x))$ for $x:X(R)$ with $U(x)=D(1)$.
\end{proof}

\begin{lemma}
  For a natural transformation $\alpha:\NatTrans{X}{Y}$, the realization
  $\abs{\alpha}$ is a morphism of locally ringed lattices, i.e.\
  for a compact open $U:\NatTrans{Y}{\mathcal{L}}$
  and function $s:\NatTrans{\coBrackets{U}}{\mathbb{A}^1}$
  we have
  \begin{align*}
    \abs{\alpha}^*\big(\Support{U}{s}\big) ~=~\Support{\abs{\alpha}^*(U)}{\abs{\alpha}^\sharp(s)}
  \end{align*}
\end{lemma}

\begin{proof}
  Unfolding the definition of $\abs{\alpha}$, we have to show that
  $\psi_U(U_s)\circ\alpha$ is equal to
  $\psi_{U\circ\alpha}\big((U\circ\alpha)_{s\circ\rest{\alpha~}{U}}\big)$
  as a natural transformation $\NatTrans{X}{\mathcal{L}}$.
  Applying either natural transformation to $x:X(R)$ with $U(\alpha(x))=D(f_1,...,f_n)$
  computes to the same value in $\ZL$, namely
  $\bigvee_{i=1}^n\psi_{f_i}\big(D(s(\alpha(x_i)))\big)$, where
  $x_i:X(\locEl{R}{f_i})$ is the canonical restriction and $\psi_{f_i}$
  is given by \cref{lem:isoLocDownSet}. This proves the desired equality.
\end{proof}

Now, let
$\mathsf{lift}:\mathsf{LRDL}_{\ell}^{op}\to\mathsf{LRDL}_{\ell+1}^{op}$
be the functor that is induced by universe lifting. Note that this
overload of notation, with $\mathsf{lift}$ being the operation on
universes as well as the induced functor for rings and locally ringed
lattices, is justified by the fact that lifting commutes with all the type
formers. This is used in the following lemma, which is needed to prove that the
realization of a functorial qcqs-scheme is a (big) qcqs-scheme.

\begin{lemma}\label{lem:isoGeoRelSpSpec}
  We have
  $\abs{\mathsf{Sp}(R)}\cong \mathsf{lift}(\mathsf{Spec}(R))\cong \mathsf{Spec}(\mathsf{lift}(R))$
  in $\mathsf{LRDL}_{\ell+1}^{op}$, for any $R:\mathsf{CommRing}_\ell$.
\end{lemma}

\begin{lemma}\label{lem:geoRelRestToCompOpenIso}
  For $U:\NatTrans{X}{\mathcal{L}}$, we have
  $\rest{\abs{X}}{U}\cong\abs{\coBrackets{U}}$ in $\mathsf{LRDL}_{\ell+1}^{op}$.
\end{lemma}

\begin{proof}
  Follows directly from \cref{lem:compOpenDownIso}.
\end{proof}

\begin{corollary}\label{cor:isSchFunSch}
  If $X$ is a functorial qcqs-scheme, then $\abs{X}$ is a qcqs-scheme in $\mathsf{LRDL}_{\ell+1}^{op}$.
\end{corollary}

Realizations of functorial qcqs-schemes are not just any kind of big qcqs-scheme,
they admit a cover by small affine schemes.

\begin{definition}\label{def:essSmallQcqsSch}
  We say that $X:\mathsf{LRDL}_{\ell+1}$ is an \emph{essentially small}
  qcqs-scheme, if there merely exists a cover $u_1,\dots,u_n:L_X$
  together with small rings $R_1,...,R_n:\mathsf{CommRing}_\ell$
  and isomorphisms $\rest{X}{u_i}\cong\mathsf{Spec}(\mathsf{lift}(R_i))$.
\end{definition}

\begin{corollary}\label{cor:isEssSmallFunSch}
  If $X$ is a functorial qcqs-scheme, then $\abs{X}$ is an essentially small qcqs-scheme.
\end{corollary}

\begin{lemma}\label{lem:makeSmallSch}
  For $X:\mathsf{LRDL}_{\ell+1}^{op}$ an essentially small qcqs-scheme
  there exists a small scheme $Y:\mathsf{QcQsSch}_\ell$
  with $X\cong\mathsf{lift}(Y)$.
\end{lemma}

\begin{proof}
  This follows from \cref{prop:gluingLatticeQcqsSchemes} and the fact that
  the functor $\mathsf{lift}$ preserves colimits.
\end{proof}

\begin{remark}\label{rem:choiceSmallSch}
  Note that we did not require \emph{mere} existence of an
  essentially small scheme in \cref{lem:makeSmallSch}.  This is
  because the $\Sigma$-type of $Y:\mathsf{QcQsSch}_\ell$ with an
  isomorphism $X\cong\mathsf{lift}(Y)$ is a proposition, since locally
  ringed lattices form a univalent category by
  \cref{prop:isUnivalentLRDL} and $\mathsf{lift}$ is fully faithful.
  In other words, \cref{lem:makeSmallSch} allows us to \emph{choose}
  an isomorphic small scheme for each essentially small scheme. This choice can
  even be made functorial by invoking the chosen isomorphisms,
  giving us a \emph{functor} from essentially small to small schemes.\footnote{This is
  basically the trick used in the proof of \cite[Lemma 9.4.5]{HoTTBook}.}
  This functor appears implicitly in the proof of \cref{thm:comparisonThm}.
\end{remark}

\begin{proposition}\label{prop:isFullyFaithfulGeoRel}
  The functor $\abs{\_}:\mathsf{FunnQcQsSch}_\ell\to\mathsf{QcQsSch}_{\ell+1}$
  is fully faithful.
\end{proposition}

\begin{proof}
  Let $X,Y$ be functorial qcqs-schemes. Any
  $\pi:\catHom{\mathsf{LRDL}_{\ell+1}^{op}}{\abs{X}}{\abs{Y}}$
  is locally a morphism of small affine schemes in the following sense:
  There merely exist compatible affine compact open covers
  $U_1,...,U_n:\NatTrans{X}{\mathcal{L}}$ with $\coBrackets{U_i}\cong\mathsf{Sp}(A_i)$
  and $V_1,...,V_m:\NatTrans{Y}{\mathcal{L}}$ with $\coBrackets{V_j}\cong\mathsf{Sp}(B_j)$,
  where the $A_i$ and $B_j$ are small rings.
  Compatibility means that for each $U_i$ exists a $V_j$ such that $U_i\leq\pi^*(V_j)$
  and the ``restriction''
  $\rest{\pi}{U_i}:\catHom{\mathsf{LRDL}_{\ell+1}^{op}}{\rest{\abs{X}}{U_i}}{\rest{\abs{Y}}{V_j}\!\!}$
  is induced by a ring morphism $\varphi_{ij}:\mathsf{Hom}(B_j,A_i)$.
  This means that $\mathsf{lift}(\varphi_{ij})$ agrees with
  $\rest{\pi^\sharp}{U_i}:\mathsf{Hom}\big(\mathcal{O}(\coBrackets{V_j}),\mathcal{O}(\coBrackets{U_i})\big)$,
  modulo the canonical $\mathcal{O}(\coBrackets{V_j})\cong\mathsf{lift}(B_j)$
  and $\mathcal{O}(\coBrackets{U_i})\cong\mathsf{lift}(A_i)$ given by \cref{prop:OSpRelCoadj},
  and that by applying
  $\mathsf{Spec}$ to this morphism we recover $\rest{\pi}{U_i}$.
  Note that such $\varphi_{ij}$ exist since $\mathsf{lift}$ is fully faithful.

  The morphism $\varphi_{ij}$ induces a natural transformation
  $\alpha_i:\NatTrans{\coBrackets{U_i}}{\NatTrans{\coBrackets{V_j}}{Y}}$,
  which we can glue to an $\alpha:\NatTrans{X}{Y}$ by \cref{lem:funMorphismGluing}.
  Cumbersome computations show that, modulo the
  equivalence $\rest{\abs{X}}{U_i}\cong\abs{\coBrackets{U_i}}$
  of \cref{lem:geoRelRestToCompOpenIso}, $\abs{\alpha_i}$ is just
  $(\_\wedge V_j)\circ\rest{\pi}{U_i}:\catHom{\mathsf{LRDL}_{\ell+1}^{op}}{\rest{\abs{X}}{U_i}}{\abs{Y}}$,
  as both maps are induced by $\varphi_{ij}$.
  In other words, the following diagram commutes
  \[\begin{tikzcd}
          {\rest{\abs{X}}{U_i}} && {\rest{\abs{Y}}{V_j}} & {\abs{Y}} \\
          {\mathsf{Spec}\big(\mathsf{lift}(A_i)\big)} && {\mathsf{Spec}\big(\mathsf{lift}(B_j)\big)} \\
          {\abs{\coBrackets{U_i}}} && {\abs{\coBrackets{V_j}}} & {\abs{Y}}
          \arrow["{\rest{\pi~}{U_i}}", from=1-1, to=1-3]
          \arrow["\vcong"', no head, from=1-1, to=2-1]
          \arrow["\_\wedge V_j", from=1-3, to=1-4]
          \arrow["\vvcong", no head, from=1-3, to=2-3]
          \arrow[Rightarrow, no head, from=1-4, to=3-4]
          \arrow["{\mathsf{Spec}(\mathsf{lift}(\varphi_{ij}))}", from=2-1, to=2-3]
          \arrow["\vcong"', no head, from=2-1, to=3-1]
          \arrow["\vvcong", no head, from=2-3, to=3-3]
          \arrow[from=3-1, to=3-3]
          \arrow[from=3-3, to=3-4]
          \arrow["{\abs{\alpha_i}}"', shift right, curve={height=12pt}, from=3-1, to=3-4]
  \end{tikzcd}\]
  This shows that $\alpha$ is the
  unique natural transformation such that $\abs{\alpha}=\pi$.
\end{proof}

\subsection{Comparison theorem} %%%%%%%%%%%%%%%%%%%%%%%%%%%%%%%%%%%%%%%%%%%%%%%%%%%%%%%%%%%%%%%%%%%%%%%%%%%%%

Morally, the functor of points and the geometric realization functor
give us an equivalence of the two notions of qcqs-scheme.
However, in our setting they only form a relative coadjunction.

\begin{proposition}\label{prop:relAdjSch}
  $\abs{\_}\dashv_{\mathsf{lift}}h$.
\end{proposition}

\begin{proof}
  Let $X:\ZFunctor_\ell$ and $Y:\mathsf{LRDL}_{\ell}^{op}$.
  We give a bijection
  \begin{align*}
    \_^\flat:\catHom{\mathsf{LRDL}_{\ell+1}^{op}}{\abs{X}}{\mathsf{lift}(Y)}~\cong~\big(\NatTrans{X}{h_Y}\big)
  \end{align*}
  For the left to right direction, let
  $\pi:\catHom{\mathsf{LRDL}_{\ell+1}^{op}}{\abs{X}}{\mathsf{lift}(Y)}$ be
  a morphism of (big) locally ringed lattices.
  For a $R$-valued point $\alpha:\NatTrans{\mathsf{Sp}(R)}{X}$
  we take its geometric realization and compose with $\pi$ to get
  $\vert\alpha\vert\circ\pi:\catHom{\mathsf{LRDL}_{\ell+1}^{op}}{\abs{\mathsf{Sp}(R)}}{\mathsf{lift}(Y)}$.
  By \cref{lem:isoGeoRelSpSpec} and the properties of $\mathsf{lift}$, we get a chain of equivalences
  \begin{align*}
    \catHom{\mathsf{LRDL}_{\ell+1}^{op}}{\abs{\mathsf{Sp}(R)}}{\mathsf{lift}(Y)}
    ~&\cong~
    \catHom{\mathsf{LRDL}_{\ell+1}^{op}}{\mathsf{Spec}(\mathsf{lift}(R))}{\mathsf{lift}(Y)} \\
    ~&\cong~ \catHom{\mathsf{LRDL}_{\ell}^{op}}{\mathsf{Spec}(R)}{Y}
  \end{align*}
  Applying these equivalences to $\vert\alpha\vert\circ\pi$
  gives the desired morphism $\pi^\flat(\alpha):h_Y(R)$.
  This is immediately seen to be natural in $R$, i.e.\ define the
  required natural transformation.

  Now, let $\alpha:\NatTrans{X}{h_Y}$ be a natural transformation.
  We need to define a morphism of locally ringed lattices
  $\pi:\equiv(\pi^*,\pi^\sharp):\catHom{\mathsf{LRDL}_{\ell+1}^{op}}{\abs{X}}{\mathsf{lift}(Y)}$.
  Let $u:L_Y$, this defines a natural transformation
  \begin{align*}
    &\theta_u:\NatTrans{h_Y}{\mathcal{L}} \\
    &\theta_u(\rho^*,\rho^\sharp) ~:\equiv~\rho^*(u)
  \end{align*}
  We can then set $\pi^*(u):\equiv\theta_u\circ\alpha:\NatTrans{X}{\mathcal{L}}$.
  We now need to provide a morphism of rings
  $\pi^\sharp:\mathcal{O}_Y(u)\to\big(\NatTrans{\coBrackets{\pi^*(u)}}{\mathbb{A}^1}\big)$.
  So let us fix $s:\mathcal{O}_Y(u)$ and $x:X(A)$ such that
  $\theta_u(\alpha(x))\equiv\alpha(x)^*(u)=D(1)$.
  This means that we have a morphism
  \begin{align*}
    \alpha(x)^\sharp:\mathcal{O}_Y(u)\to \mathcal{O}_A(D(1)) \quad(=A)
  \end{align*}
  which, modulo the above identification, lets us define
  $\pi^\sharp(s)(x):\equiv\alpha(x)^\sharp(s)$.
  We omit the proof that this is a ring homomorphism natural in $A$.
  In order to check that $\pi$ is a morphism of locally ringed lattices,
  let $u:L_Y$ and $s:\mathcal{O}_Y(u)$. We need give an equality
  of natural transformations $\NatTrans{X}{\mathcal{L}}$, so let $x:X(R)$
  be given such that $\alpha(x)^*(u)=D(f_1,...,f_n)$.
  We get
  \begin{align*}
    \pi^*\big(\Support{u}{s}\big)(x) ~&=~ \alpha(x)^*\big(\Support{u}{s}\big) \\ &=~ \Support{\alpha(x)^*(u)}{\alpha(x)^\sharp(s)} \\ &=~ {\textstyle \bigvee_{i=1}^n} \Support{D(f_i)}{\rest{\alpha(x)^\sharp(s)}{D(f_i)}} \\ &=~ {\textstyle \bigvee_{i=1}^n} \psi_{f_i}\big(D(\alpha(x_i)^\sharp(s))\big) \\&=~ \Support{\pi^*(u)}{\pi^\sharp(s)}(x)
  \end{align*}
  where $x_i:X(\locEl{R}{f_i})$ is the canonical restriction and
  $\psi_{f_i}$ is given by \cref{lem:isoLocDownSet}.  We omit the
  proof that this is inverse to $\_^\flat$.
\end{proof}

\begin{corollary}\label{cor:flatPresIso}
  For a functorial qcqs-scheme $X$ and $Y:\mathsf{LRDL}_{\ell}^{op}$, the map
  \begin{align*}
    \_^\flat:\catHom{\mathsf{LRDL}_{\ell+1}^{op}}{\abs{X}}{\mathsf{lift}(Y)}\to\big(\NatTrans{X}{h_Y}\big)
  \end{align*}
  given in the proof of \cref{prop:relAdjSch},
  maps isomorphisms of ringed lattices to natural isomorphisms.
\end{corollary}

\begin{proof}
  Let $\pi:\catHom{\mathsf{LRDL}_{\ell+1}^{op}}{\abs{X}}{\mathsf{lift}(Y)}$ be
  an isomorphism of locally ringed lattices and $R:\mathsf{CommRing}_\ell$.
  Since $\abs{\_}$ is fully faithful on functorial qcqs-schemes
  by \cref{prop:isFullyFaithfulGeoRel}, the map
  \begin{align*}
    \abs{\_}\circ\pi :\big(\NatTrans{\mathsf{Sp}(R)}{X}\big) \to \catHom{\mathsf{LRDL}_{\ell+1}^{op}}{\abs{\mathsf{Sp}(R)}}{\mathsf{lift}(Y)}
  \end{align*}
  is a bijection.
\end{proof}

\begin{theorem}\label{thm:comparisonThm}
  We have an adjoint equivalence of categories between functorial and geometrical qcqs-schemes.
\end{theorem}
\begin{proof}
  Since locally ringed lattices and $\Z$-functors form univalent
  categories (\cref{prop:isUnivalentLRDL} and \cite[Thm.\ 9.2.5]{HoTTBook}),
  it suffices  by \cite[Lemma 9.4.7]{HoTTBook} to show that
  $h_{(\_)}:\mathsf{QcQsSch}_\ell\to\mathsf{FunQcQsSch}_\ell$
  is fully faithful and essentially surjective. Fully faithfulness was already proved in
  \cref{prop:isFullyFaithfulh}. Let $X$ be a functorial qcqs-scheme.
  Looking at the proof of \cref{cor:isSchFunSch},
  we can see that $\abs{X}$ is an essentially small scheme.
  By \cref{lem:makeSmallSch}, we can thus assume the existence of a $Y:\mathsf{QcQsSch}_\ell$
  with an isomorphism $\abs{X}\cong\mathsf{lift}(Y)$.
  Using \cref{cor:flatPresIso} this gives us the desired natural isomorphism $X\cong h_Y$.
\end{proof}

\begin{remark}\label{rem:comparisonDG}
  We want to conclude this section with a comparison of our approach with
  the presentation in \cite{DemazureGabriel}.
  Demazure and Gabriel assume two Grothendieck universes
  and define $\Z$-functors to map small rings to big sets.
  To model this situation in our setting, assume cumulative universes
  $\mathsf{Type}_\ell:\mathsf{Type}_{\ell+1}$ and define $\Z$-functors
  to be functors $\mathsf{CommRing}_\ell\to\mathsf{Set}_{\ell+1}$.
  The spectrum $\mathsf{Sp}(A):\equiv\mathsf{Hom}(A,\_)$ can then even
  be defined for big $A:\mathsf{CommRing}_{\ell+1}$. Note however,
  that this $\mathsf{Sp}$ is only fully faithful when restricted to
  $\mathsf{CommRing}_\ell^{op}$. This caveat and the fact that
  universes are simply not cumulative in \texttt{Cubical Agda},
  led the formalization \cite{ZeunerHutzler24} to deviate from \cite{DemazureGabriel}.
  We decided to follow the formalization here, as one of the main objectives
  was to develop a univalent and constructive theory of schemes that
  can be formalized in e.g.\ \texttt{Cubical Agda}.

  The advantage of the approach of \cite{DemazureGabriel} is that it allows one to avoid
  relative coadjunctions. The ``big spectrum functor'' $\mathsf{Sp}$ is right adjoint to
  $\mathcal{O}$. Similarly, we can define define the functor of points
  \begin{align*}
    h_X:\equiv \catHom{\mathsf{LRDL}^{op}_{\ell+1}}{\mathsf{Spec}(\_)}{X}~:~\mathsf{CommRing}_\ell\to\mathsf{Set}_{\ell+1}
  \end{align*}
  for any big $X:\mathsf{LRDL}_{\ell+1}$ and we get an adjunction
  $\abs{\_}\dashv h$.  We can then show that this becomes an adjoint
  equivalence when restricted to functorial qcqs-schemes and
  essentially small qcqs-schemes.  In this setting, one can show the
  existence of the geometric realization functor $\abs{\_}$ using
  purely categorical methods, which further simplifies the proof.
  Note that the category of functors
  $\mathsf{CommRing}_\ell\to\mathsf{Set}_{\ell+1}$ and
  $\mathsf{LRDL}^{op}_{\ell+1}$ have colimits
  over diagrams in $\mathsf{Type}_{\ell+1}$.\footnote{Showing the completeness of $\mathsf{LRDL}$
    is straightforward, perhaps even more so than the cocompleteness of
    locally ringed spaces. Everything, including the support, is computed point-wise.}
  We can define $\abs{\_}$ to be the \emph{left Kan-extension} of the
  restricted
  $\mathsf{Spec}:\mathsf{CommRing}_\ell^{op}\to\mathsf{LRDL}^{op}_{\ell+1}$
  along the restricted $\mathsf{Sp}$, i.e.\ the Yoneda
  embedding. By a standard result, it is the unique functor (up to
  unique natural isomorphism) extending the restricted
  $\mathsf{Spec}$ that is cocontinuous and left-adjoint to
  $h$.\footnote{See e.g.\ section 1.3 of \cite{Vezzani}.}

  In our predicative, constructive setting we do not expect to be able
  to show the existence of the Kan extension by general methods, as
  $\ZFunctor_\ell$ does not have colimits in $\mathsf{Type}_{\ell+1}$.
  Nevertheless, it is still possible to prove the comparison theorem
  using the explicit point-free description of $\abs{\_}$ as shown in
  this paper. Univalence allows us two view $\abs{\_}$ as a sort of
  adjoint inverse to $h$ on (functorial) qcqs-schemes even with a
  mismatch of universe levels present. For a functorial qcqs-scheme $X$,
  $\abs{X}$ has a cover by small affine schemes, and as sketched in \cref{rem:choiceSmallSch},
  we can  choose a small scheme for $\abs{X}$ in a functorial way.
  In other words the adjoint inverse to $h$
  is $\abs{\_}$ composed with the functor choosing a small scheme
  for an essentially small scheme.
\end{remark}

\section{Schemes of finite presentation}\label{sec:SchemesOfFP}

For the statement and proof of the comparison theorem we required
three universes
$\mathsf{Type}_\ell\hookrightarrow\mathsf{Type}_{\ell+1}\hookrightarrow\mathsf{Type}_{\ell+2}$
(the last one being the universe in which $\mathsf{LRDL}_{\ell+1}$ lives).
The comparison theorem presented in the previous section is predicative in that it does
not require resizing assumptions. From a truly predicative point of view, however,
having to work in $\mathsf{Type}_{\ell+1}$ for a considerable part of the proof
in order to prove the equivalence for schemes in $\mathsf{Type}_\ell$
feels somewhat unsatisfactorily ad hoc.\footnote{Perhaps worse still,
  in a certain sense the types of algebraic structures like groups and
  rings depend on the universe.  This can be shown already in a rather
  minimal type theory. The special case of groups is discussed in
  \cite{MoreGroups}.}
We thus want to outline a more ``parsimonious'' approach that requires
only one universe, in which a specified but arbitrary ring $R$ is
assumed and an ambient universe where e.g.\ the category of rings
(that $R$ belongs to) lives. The caveat of this approach is that we
can only describe schemes of finite presentation over $R$, a subclass
of qcqs-schemes.  However, these schemes still include many motivating
examples in algebraic geometry that come from studying roots of
polynomials. Blechschmidt considers the ``parsimonious site'' of
affine schemes that are locally of finite presentation over some base
scheme $S$, which can be constructed without size issues
\cite[Sec.\ 15]{BlechschmidtPhD}. In the special case where
$S=\mathsf{Spec}(R)$ is the affine scheme corresponding to our fixed
ring $R$, one can define schemes of finite presentation over $R$ in a
purely functorial way \cite[Sec.\ 16.5]{BlechschmidtPhD}.\footnote{These are the
  external counterpart for Blechschmidt's notion of
  ``finitely presented synthetic scheme'' \cite[Def.\ 19.49]{BlechschmidtPhD}.}
The key algebraic notion is that of a finitely presented algebra.

\begin{definition}\label{def:fpAlgebra}
  We say that an $R$-algebra $A$ is \emph{finitely presented} if
  there merely exist polynomials $p_1,...,p_m:R[x_1,...,x_n]$ together with an isomorphism
  $A\cong\nicefrac{R[x_1,...,x_n]}{\langle p_1,...,p_m\rangle}$ of $R$-algebras.
  A ring morphism $\mathsf{Hom}(R,A)$ is \emph{of finite presentation} if it makes $A$ into a
  finitely presented $R$-algebra. For a ring $R$, the category $\fpAlg{R}$
  has objects pairs of positive integers $n,m:\mathbb{N}$
  together with a list of polynomials $p_1,...,p_m$ with $p_i:R[x_1,...,x_n]$.
  Such an object $(p_1,...,p_m)_{n,m}$ represents the finitely presented algebra
  $\nicefrac{R[x_1,...,x_n]}{\langle p_1,...,p_m\rangle}$.
  Accordingly, arrows are $R$-algebra morphisms, i.e.\
  \begin{align*}
    \catHom{\fpAlg{R}&}{(p_1,...,p_m)_{n,m}}{(q_1,...,q_l)_{k,m}}  \\&:\equiv \catHom{\mathsf{Hom}_R}{\nicefrac{R[x_1,...,x_n]}{\langle p_1,...,p_m\rangle}}{\nicefrac{R[x_1,...,x_k]}{\langle q_1,...,q_l\rangle}}
  \end{align*}
  This way, we can present the category of finitely presented $R$-algebras as a small category
  living in the same universe as $R$.\footnote{Note that this category is not univalent.
    However, this does not affect any of the following results.
    One could of course consider the Rezk completion \cite[Sec.\ 9.8]{HoTTBook},
    but since we only consider functors from $\fpAlg{R}$ to univalent categories,
    these will always factor through the Rezk completion.}
\end{definition}

\begin{example}\label{ex:locIsFP}
    For a ring $R$ and $f:R$ the localization $\locEl{R}{f}$ is a finitely
    presented $R$-algebra, as $\locEl{R}{f}\cong\nicefrac{R[x]}{\langle fx-1\rangle}$.
\end{example}

\begin{lemma}\label{lem:tensorProdIsFP}
  If $A$ and $B$ are two finitely presented $R$-algebras, the
  tensor product $A\otimes_R B$ is finitely presented.
\end{lemma}

\begin{lemma}\label{lem:homOfFP}
  Let $A,B,C$ be rings and $\varphi:\mathsf{Hom}(A,B)$ and $\psi:\mathsf{Hom}(B,C)$
  be morphisms of rings. Then the following hold:
  \begin{enumerate}
  \item If $\varphi$ and $\psi$ are of finite presentation,
    then $\psi\circ\varphi$ is of finite presentation.
  \item If $\varphi$ and $\psi\circ\varphi$ are of finite presentation,
    then $\psi$ is of finite presentation.
  \end{enumerate}
\end{lemma}

\begin{proof}
  \cite[\href{https://stacks.math.columbia.edu/tag/00F4}{Tag 00F4}]{stacks-project}
\end{proof}

Defining schemes of finite presentation using
locally ringed lattices can be done analogously to the standard classical setting.
By a qcqs-scheme over $R$
we mean a $X:\mathsf{QcQsSch}$ together with a morphism
$\pi_X:\catHom{\mathsf{LRDL}^{op}}{X}{\mathsf{Spec}(R)}$,
and we will simply call them $R$-schemes. A morphism of $R$-schemes is a morphism
in the slice category $\nicefrac{\mathsf{LRDL}^{op}}{\mathsf{Spec}(R)}$.
For a $R$-scheme $X$ we may regard the structure sheaf as taking values in $R$-algebras,
i.e.\ $\mathcal{O}_X:L_X^{op}\to R\text{-}\mathsf{Alg}$.

\begin{definition}\label{def:schemeoffp}
  A morphism between two qcqs-schemes $\pi:X\to Y$
  is said to be \emph{of finite presentation}
  if there are compatible affine covers $u_1,\dots,u_n$ of $L_X$ and
  $w_1,\dots,w_m$ of $L_Y$ such that
  for every $u_i$ and $w_j$ with $u_i\leq\pi^*(w_j)$
  the ring morphism $\rest{\pi^\sharp}{u_i}:\mathcal{O}_Y(w_j)\to\mathcal{O}_X(u_i)$
  is of finite presentation.
  A qcqs-scheme $X$ over $R$, is said to be of \emph{finite presentation} if
  its morphism $X\to\mathsf{Spec}(R)$ is of finite presentation.
\end{definition}

Using \cref{lem:homOfFP} and some standard computations for morphisms
of schemes, one can prove the following two lemmas.

\begin{lemma}\label{lem:schemeoffp}
  A $R$-scheme $X$ is of finite presentation if and only if
  there is a an affine cover $u_1,..., u_n:L_X$, such that
  $\mathcal{O}_X(u_i)$ is a finitely presented  $R$-algebra for all $i$ in $1,...,n$.
\end{lemma}

\begin{lemma}\label{lem:morphismoffp}
  Let $X$ and $Y$ be two $R$-schemes of finite presentation, then any morphism
  of $R$-schemes $X\to Y$ is of finite presentation.
\end{lemma}

Note that even if $X$ is an $R$-scheme of finite presentation,
the global sections $\mathcal{O}_X(1)$ need not be a finitely presented $R$-algebra.\footnote{See
  e.g.\ \cite{VakilExample} for a concrete counterexample.}
We still get the following:

\begin{proposition}\label{prop:universalPropFP}
  Let $\Gamma:\nicefrac{\mathsf{LRDL}^{op}}{\mathsf{Spec}(R)}\to R\text{-}\mathsf{Alg}^{op}$
  be the global sections functor,
  $\mathsf{Spec}:\fpAlg{R}^{op}\to\nicefrac{\mathsf{LRDL}^{op}}{\mathsf{Spec}(R)}$
  be the restriction of the usual spectrum functor to finitely presented algebras
  and $\mathcal{U}:\fpAlg{R}\to R\text{-}\mathsf{Alg}$ be the forgetful functor. Then
  we have a relative coadjunction $\Gamma\dashv_{\,\mathcal{U}} \mathsf{Spec}$
  whose relative counit is an isomorphism.
\end{proposition}

Now for the functorial definition of schemes of finite presentation.
These will be \emph{finitely presented $R$-functors}
$\fpAlg{R}\to\mathsf{Set}$, i.e.\ presheaves on $\fpAlg{R}^{op}$. The
main difference to the situation of general qcqs-schemes is that this
presheaf category is locally small.
Important presheaves, like $\Aone$ and $\mathcal{L}$ are
defined in the obvious way (by precomposition with the forgetful
functor to rings).  However, functions $\NatTrans{X}{\Aone}$ and
compact opens $\NatTrans{X}{\mathcal{L}}$ of
$X:\fpAlg{R}\to\mathsf{Set}$ are now small types. Note that
$\NatTrans{X}{\Aone}$ carries an $R$-algebra structure but does not
have to be finitely presented.  Similar to the locally ringed lattice
case we get a relative coadjunction with respect to the forgetful
functor $\mathcal{U}:\fpAlg{R}\to R\text{-}\mathsf{Alg}$.

\begin{proposition}\label{prop:universalPropFunFP}
  Let $\mathsf{Sp}:\fpAlg{R}^{op}\hookrightarrow\catHom{\mathsf{Fun}}{\fpAlg{R}}{\mathsf{Set}}$
  denote the Yoneda embedding and
  $\mathcal{O}:\catHom{\mathsf{Fun}}{\fpAlg{R}}{\mathsf{Set}}\to R\text{-}\mathsf{Alg}^{op}$
  the global functions functor. Then we have a relative coadjunction
  $\mathcal{O}\dashv_{\,\mathcal{U}} \mathsf{Sp}$
  whose relative counit is an isomorphism.
\end{proposition}

The Zariski coverage on $\fpAlg{R}^{op}$ is defined just as for
$\mathsf{CommRing}^{op}$.  Note that this crucially relies on the fact
that by \cref{ex:locIsFP} localizations $\locEl{R}{f}$ and by
\cref{lem:tensorProdIsFP} tensor products are finitely presented. The
proof of subcanonicity carries over to finitely presented
algebras. Compact opens are defined just as for $\Z$-functors, leading
to a functorial definition of $R$-schemes of finite presentation.

\begin{definition}\label{def:funschemeoffp}
  We say that $X:\fpAlg{R}\to\mathsf{Set}$ is a
  \emph{functorial $R$-scheme of finite presentation} if it is local
  and merely has a cover $U_1,...,U_n:\NatTrans{X}{\mathcal{L}}$
  with $\coBrackets{U_i}\cong\mathsf{Sp}(A_i)$ for some
  finitely presented $R$-algebra $A_i$.\footnote{Note that when working with
    a non-univalent definition of $\fpAlg{R}$, a truncation is
    needed to define the property of ``being an affine $R$-scheme of finite presentation'',
    i.e.\ the property of being a representable, see \cite[Thm.\ 9.5.9]{HoTTBook},}
\end{definition}

To show that the two definitions of $R$-schemes of finite presentation are equivalent,
we proceed as for qcqs-schemes, but we do not have to be as careful about universe levels.
The functor of points
\begin{align*}
  h_{(\_)}:\equiv\catHom{\nicefrac{\mathsf{LRDL}^{op}}{\mathsf{Spec}(R)}}{\mathsf{Spec}(\_)}{X}~:~\nicefrac{\mathsf{LRDL}^{op}}{\mathsf{Spec}(R)}\to\catHom{\mathsf{Fun}}{\fpAlg{R}}{\mathsf{Set}}
\end{align*}
is defined just as for $\Z$-functors. The realization functor $\abs{\_}$ can be defined
directly following \cref{def:realizationFun} and using the fact that an object
of $\nicefrac{\mathsf{LRDL}^{op}}{\mathsf{Spec}(R)}$ is the same as a locally ringed lattice
whose structure sheaf is a sheaf of $R$-algebras. Alternatively, one can take
$\abs{\_}$ to be the left Kan extension of $\mathsf{Spec}$ along $\mathsf{Sp}$.
Using the small definition of $\fpAlg{R}$, this Kan extension can be shown to exist
predicatively, as it can be computed pointwise \cite[Thm.\ 6.3.7]{RiehlCatTheory}.

Moreover, we get a regular adjunction $\abs{\_}\dashv h$.
We can show that $h$ sends $R$-schemes of finite presentation to
functorial $R$-schemes of finite presentation, using \cref{lem:schemeoffp}.
Similarly, $\abs{\_}$ sends functorial $R$-schemes of finite presentation to
$R$-schemes of finite presentation and we do not have to consider
``essentially small schemes'' or the like. When restricted to (functorial) $R$-schemes
of finite presentation, the two adjoints become fully faithful, inducing
an equivalence of categories as illustrated in the diagram below:
\[\begin{tikzcd}
	{R\text{-}\mathsf{FunSch}_{fp}} &&&& {R\text{-}\mathsf{Sch}_{fp}} \\
	&&&& \\
	{\catHom{\mathsf{Fun}}{\fpAlg{R}}{\mathsf{Set}}} &&&& {\nicefrac{\mathsf{LRDL}^{op}}{\mathsf{Spec}(R)}} \\
	\\
	&& {\fpAlg{R}^{op}}
	\arrow["{\mathsf{Spec}}"', hook, from=5-3, to=3-5]
	\arrow["{\mathsf{Sp}}", hook', from=5-3, to=3-1]
	\arrow[""{name=0, anchor=center, inner sep=0}, "{\abs{\_}}"', curve={height=12pt}, from=3-1, to=3-5]
	\arrow[""{name=1, anchor=center, inner sep=0}, "{h_{(\_)}}"', curve={height=12pt}, from=3-5, to=3-1]
	\arrow[hook', from=1-1, to=3-1]
	\arrow["\Large\simeq"{description}, no head, from=1-1, to=1-5]
	\arrow[hook, from=1-5, to=3-5]
	\arrow["\dashv"{anchor=center, rotate=90}, draw=none, from=0, to=1]
\end{tikzcd}\]

\section{Conclusion}

In this paper we gave two point-free definitions of qcqs-schemes,
namely as locally ringed lattices and as $\Z$-functors, and proved the
two notions equivalent.  We worked constructively in HoTT/UF using
univalence and higher inductive types. Due to size issues inherent in
the functor of points approach, we had to work across several
universes, but we did stay predicative in the sense that we did not
require any form of propositional resizing.  In the last section, we
described how one can then define schemes of finite presentation over
an arbitrary base ring and prove a similar equivalence result, but
with less universes.

Even though univalence did play a crucial role in the proof of the
main comparison theorem, many of the results in this paper should be
of interest for the program of constructive algebraic geometry in
general.  With the constructive definition (or rather the definitions)
of scheme pinned down, there are plenty of directions in which
constructive algebraic geometry could be developed. We do not want to
engage in speculation about which parts of this vast field can be
constructivized. We want to stress, however, that because of the
computational content of constructive proofs, it would be particularly
interesting to see whether a constructive development can be leveraged
for computational problems in algebraic geometry.  The relation to
synthetic algebraic geometry needs also to be explored.  Especially in
view of the recent use of HoTT/UF for this internal
approach by Cherubini, Coquand and Hutzler \cite{SAG}, one might hope
that a comparison with our external approach can lead to
fruitful insights.\footnote{For an overview of the many (sub-) topics
  currently investigated in this fascinating young research program see
  \url{https://github.com/felixwellen/synthetic-zariski}}
Finally, it is worth noting that the fundamental
notions of constructive scheme theory, as presented in this paper,
allow for a concise and pleasant presentation. Hopefully,
this will serve an inspiration to further the development
of constructive algebraic geometry, with or without univalence.

The other main objective was of course to give an account of basic
scheme theory that can be formalized in a proof assistant like
\texttt{Cubical Agda}. While the paper can be used as a general
blueprint for a formalization of the comparison theorem, the same cannot
be said for some of the individual proofs and definitions given in
the paper.  Extending the formalization of functorial qcqs-schemes
in \cite{ZeunerHutzler24} to include all results of
\cref{sec:fopApproach} should be more or less directly feasible.  A
formalization of qcqs-schemes as locally ringed lattices, however,
should be regarded as a substantial formalization project, even with
the structure sheaf of an affine scheme already formalized in
\cite{ZeunerMortberg23}.  This is because we did not work with an exact
definition of the structure sheaf and instead relied heavily on
canonical isomorphisms.  The case that this common practice generally
creates an obstacle towards formalization was put forward rather
convincingly by Buzzard \cite{BuzzardEquality}.  In view of the
results of \cite{ZeunerMortberg23}, one might hope that in this
particular instance univalence could be of help, but as indicated in
the paper, careful and clever rephrasing of many of the results and
constructions might be required.  Ultimately, it appears necessary to
prove a lot of results on locally ringed lattices in terms of sheaves
obtained by Kan extensions, before even being able to describe affine
schemes as locally ringed lattices. Formalizing the
proof of the univalent constructive comparison theorem will
certainly hold many surprises.  All in all, such a formalization would
figure as a major milestone in formal univalent set-level mathematics.

\bibliographystyle{alphaurl}
\bibliography{references}
\end{document}